\documentclass[12pt]{article}

\usepackage{amsfonts,amsmath,amsthm,amssymb}
\usepackage{enumerate}
\usepackage{float}
\usepackage{epic,eepic}
\usepackage[dvips]{color}
\usepackage{graphicx}
\usepackage{hyperref}
\usepackage{authblk}

\newtheorem{thm}{Theorem}[section]

\newtheorem{clm}[thm]{Claim}
\newtheorem{cor}[thm]{Corollary}
\newtheorem{lem}[thm]{Lemma}

\newtheorem{rem}[thm]{Remark}

\newtheorem{df}[thm]{Definition}

\newtheorem{exm}[thm]{Example}

\newenvironment{fg}{\begin{figure}}{\end{figure}}
\newenvironment{enum}{\begin{enumerate}}{\end{enumerate}}
\newenvironment{eq}{\begin{equation}}{\end{equation}}

\newcommand{\lra}{\longrightarrow}

\newcommand{\del}{\partial}
\newcommand{\mcal}{\mathcal}

\newcommand{\are}{\mcal{A}}
\newcommand{\len}{\ell}

\newlength{\standardunitlength}
\begin{document}
\title{Quasi Normality and $PL$ Approximation of Least Area Surfaces in $3$-Manifolds}
\author{Eli Appleboim\\
\href{mailto:eliap@ee.technion.ac.il}{eliap@ee.technion.ac.il}}
\affil{Technion, Haifa, Israel \\ Gordon Academic College of Education, Haifa, Israel}

\date{}

\maketitle
\begin{abstract}
{\noindent This} paper presents relations between least area and normal surfaces, embedded in either a Euclidean or hyperbolic $3$-manifold. A relaxed version of normal surfaces, termed quasi-normal, is introduced, and it is shown that under appropriate conditions, every embedded least area surface is quasi-normal with respect to a fine enough fat triangulation of the $3$-manifold. In addition, it is shown that the intersections of a least area surface with the tetrahedra of such fine enough triangulation, even when not as simple as in the case of normal surfaces, are also well behaved. Finally, it is shown that a least area surface, when considered as a quasi normal surface, gives rise to a sequence of piecewise flat surfaces termed as flat-associated surfaces, and this sequence converges to the given least area surface and approximates its area.  
\end{abstract}

\section{Introduction}
{\noindent Let} $(\mcal{M}, \mathfrak{g}, \mcal{T})$ be a smooth $3$-manifold equipped with a Riemannian metric $\mathfrak{g}$ and a triangulation $\mcal{T}$. For surfaces embedded in $\mcal{M}$ there are two notions of area that can be associated. One is the usual notion of area that is given by the metric $\mathfrak{g}$. The second version is given by the intersection of a surface with $\mcal{T}$, and can be considered as a topological version of area. For each area function there is a notion of simplicity of surfaces. In the metric case this is the notion of least area surfaces, where a least area surface minimizes the area in some equivalence class, while in the topological view point it is given by least weight normal surfaces for which their intersection pattern with the triangulation admits some minimality criteria.\\
{\noindent Given} these two versions of area minimizing surfaces it is most natural to look for possible relations between them. It is already known that these two families of surfaces share some common properties. For instance, in \cite{msy} it is proved that every incompressible surface can be isotoped to a least area surface, while in \cite{jr1} it is shown that every incompressible surface can be isotoped to a least weight normal surface. Yet, a complete picture of the relations between the two types of area minimizing surfaces is far from being achieved. For example, in \cite{jr2} the following, still open, conjecture is made: "... it should be possible to subdivide a triangulation arbitrarily finely and obtain sequences of $PL$ minimal surfaces which converge to classical analytic minimal surfaces. This would be analogous to some of the early work on the Dirichlet principle." (see, \cite{jr2}, pp.495-496). A possible rephrase of this conjecture is given by the two following questions \\

{\noindent (QI)} Is it possible to subdivide a given triangulation of a $3$-manifold arbitrarily fine, so that every least area surface becomes a normal surface with respect to a fine enough triangulation?\\

{\noindent (QII)Is every least area surface the limit surface, in some appropriate manner, of a sequence of normal surfaces? \\

{\noindent An} affirmative answer to the first question gives a topological interpretation of the differential geometric picture, while an affirmative answer to the second question gives an alternative topological proof of many classical results about the existence of least area surfaces in $3$-manifolds, such as those that were obtained in ~\cite{msy}, using partial differential equations. Moreover, the use of existing algorithms for finding normal surfaces (such as ~\cite{jo}) makes this topological counterpart potentially computable.\\

{\noindent The} motivation for answering these two questions comes in two flavors. One is purely theoretical, as answering these questions significantly extends the analogy between minimal and normal surfaces. The second motivation is of an applied nature. For example, in the area of image processing, one of the possible ways of dealing with noisy images is by minimal surface theory(\cite{ms}, \cite{apt}). However, all existing techniques in this direction use at least second order derivations of the given noisy image, thus induce numerical instabilities that somehow need to be overcome. In contrast, the topological version of minimal and least area surfaces reflected by normal surfaces as proposed in \cite{jr2} does not rely on derivatives hence may overcome numeric issues.\\

\noindent{This} paper addresses the two questions presented above. It is organized as follows: Section $2$ is devoted to the presentation of all needed background material, covering basic notations, definitions and necessary known results on $3$-manifold topology, least area surfaces, normal surface theory and fat triangulations. Afterwards, in Sections $3$ and $4$ question (QI) is concerned. Section $3$ begins with the definition of a quasi-normal surface which can be thought of as a relaxed version of a normal surface. Then, it is shown that every least area surface is quasi normal with respect to small enough (in terms of mesh-size), fat triangulation. Moreover, the upper bound on the mesh-size will be shown to depend only on the injectivity radius of $\mcal{M}$, and on the Schoen curvature bound of least area surfaces, thus making this bound global (Section $3$, Theorem \ref{thm:LeastAreaIsNormal}). In Section $4$ a tameness theorem will be proved, asserting that the intersection patterns of a quasi-normal surface with a small enough fat triangulation are well behaved and controlled (Section $4$, Theorem \ref{thm:NonNormalAreTame}). Results in these two sections will be shown first for closed $3$-manifolds, after which necessary modifications will be done for compact $3$-manifolds with boundary (Section $4$, Theorem \ref{thm:LeastAreaQuasiNormalBoundary}, and Theorem \ref{thm:NonNormalAreTameBoundary}).

\noindent{In} Sections $5$, and $6$ the focus is turned toward question (QII). In $5$ the median subdivision of a triangulation is introduced, and  it will be shown that the median subdivision has the same fatness as the original triangulation. This yields a procedure under which a given triangulation can be iteratively subdivided, to obtain a sequence of triangulations having decreasing mesh and fixed fatness (Section $5$, Lemmas \ref{lem:MadianSubdivKeepsFatness} and \ref{lem:MedianSubdivPreservesFatness}). It is important to note that since it is necessary to keep the fatness fixed, such a subdivision process needs special care. For example, the commonly used barycentric subdivision will not do. Section $6$ makes use of all results obtained in its preceding sections in order to give a method by which it is possible to approximate a given least area surface by a sequence of simple surfaces. First, the exact meaning of convergence of a sequence of surfaces is defined. Then, as a least area surface is assumed to be in quasi normal position with respect to a given triangulation, the notion of {\em flat associate} surface of the least area surface is defined. This will follow by showing that the median subdivision procedure built in Section $5$ gives rise to a sequence of flat associated surfaces which converges to the given least area surface (Section $6$, Theorem \ref{thm:AppxThm}, and Theorem \ref{thm:AppxThmLinear}).

\section{Preliminaries}
{\noindent Throughout} this paper we work in the smooth category. The notation $\mcal{M}$ will denote a compact Riemannian $3$-manifold, the boundary of which, if exists, will be denoted by $\del \mcal{M}$. A map $f:(\mcal{F}, \del \mcal{F}) \lra (\mcal{M}, \del \mcal{M})$ between manifolds with boundary is {\it proper} if $f^{-1}(\del \mcal{M}) = \del \mcal{F}$ and the preimage of every compact subset of $\mcal{M}$ is a compact subset of $\mcal{F}$. Unless otherwise stated all concerned surfaces in $\mcal{M}$ are assumed to be properly  embedded. It will be assumed that $\mcal{M}$ is equipped with a triangulation, denoted by $\mcal{T}$. The $i$ skeleton of $\mcal{T}$ is $\mcal{T}^{(i)}, i = 0, 1, 2, 3$, i.e.
\[ \mcal{T}^{(i)} = \bigcup \{ \textrm{simplices of dimension} \leq i \} \]
{\noindent Every} triangulation of a three-manifold is assumed to be {\it non-singular} and {\it proper}, where non-singular means that every $i$-cell, $i = 0, ..., 3$, is a homeomorphic image of a standard $i$-cell in $\mathbb{R}^3$, and proper means that the intersection of two $j$-cells is either empty or exactly one $(j-1)$-cell, $j = 1, 2, 3$. It is a well known result of Moise that every smooth $3$-manifold can be non-singularly and properly triangulated, see \cite{moi}. If $\gamma$ is a curve embedded in $\mcal{M}$ we denote its length with respect to the metric by $\len(\gamma)$, and for an embedded surface $\mcal{F}$, its area is denoted by $\are(\mcal{F})$. We denote by $\lambda(\mcal{T})$ the mesh of a triangulation $\mcal{T}$, i.e. the maximal diameter of a $3$-simplex of $\mcal{T}$.

\subsection{Basics of $3$-Manifold Topology}\label{sec:PrelimTopo}
{\noindent We} give a very short and essential list of definitions and results about $3$-manifolds topology that are necessary for the sequel, for details see \cite{hem}.

{\noindent Let} $\mcal{S}: \mcal{F} \rightarrow \mcal{M}$ be an embedding of a surface $\mcal{F}$ into $\mcal{M}$. Abusively, whenever no confusion is expected, we denote the image $\mcal{S}(\mcal{F})$ by $\mcal{F}$. Two properly embedded surfaces $\mcal{F}$ and $\mcal{G}$ are {\em isotopic} if there exists a homotopy $H: I \times \mcal{M} \rightarrow \mcal{M}$ such that for every time $t \in I$ the map $H_t: \mcal{M} \rightarrow \mcal{M}$ is a homeomorphism, and so that,
\[H_0(\mcal{F}) = \mcal{F} \;\;,\]
and,
\[H_1(\mcal{F}) = \mcal{G} \;\;.\]

{\noindent A} $2$-sphere embedded in a $3$-manifold is {\em essential} if it does not bound a $3$-ball. A $3$-manifold $\mcal{M}$ is {\em irreducible} if it does not contain any essential $2$-sphere. Let $\mcal{F}$ be a surface embedded in $\mcal{M}$, of genus $> 0$. A {\em compressing disk} for $\mcal{F}$ is a disk, $\mcal{D}$, embedded in $\mcal{M}$, such that $ \partial \mcal{D} = \mcal{D} \cap \mcal{F}$, and $\partial \mcal{D}$ does not bound a disk in $\mcal{F}$. A surface $\mcal{F}$ is {\em compressible} in $\mcal{M}$ if there exists a compressing disk for $\mcal{F}$. Otherwise it is {\em incompressible}. If $\mcal{F}$ is a $2$-sphere then it is incompressible if it is essential.

{\noindent A} compact irreducible orientable $3$-manifold that contains a two sided incompressible surface is called {\em Haken manifold}. The complement of a link in $\mathbb{S}^3$ is a Haken manifold where each of the boundary torii is an embedded incompressible surface.

\subsection{Least Area Surfaces}\label{sec:PrelimMinimal}
In this subsection essential material on minimal and least area surfaces is reviewed. References to this include \cite{my1}, \cite{my2}, \cite{my3}. 
\subsubsection{Plateau Problem}
\noindent{The} Plateau problem was first set by Lagrange although named after Joseph Plateau who was interested in soap films.
The statement of the problem is as follows:\\
Given a closed piecewise smooth curve $\Gamma$ in a $3$-manifold, does there always exist a surface $\mcal{F}$ so that $\partial\mcal{F} = \Gamma$ and
\[ \are(\mcal{F}) = \inf\{ \are(\mcal{G})\; ; \partial\mcal{G} = \Gamma, and \textrm{ $\mcal{G}$ is a surface isotopic to $\mcal{F}$ }\} \]
{\noindent A} similar question can also be formulated for closed surfaces. A surface $\mcal{F}$ that admits the above equality is called a {\em least area surface}.\\

\subsubsection{Some Riemannain Geometry}
Let $\mcal{F} : \Sigma \rightarrow \mcal{M}$ be an embedding of a smooth surface $\Sigma$ into $\mcal{M}$. Let $(u, v) \in \Omega \subset \mathbb{R}^2$ denote the parametrization domain of $\Sigma$. 

{\noindent The} area of $\mcal{F}$, with respect to an area element $d\are = dudv$ that is induced by the embedding, is given by
\begin{eq}\label{eq:AreaOfSurface}
\are(\mcal{F}) = \int_{\Omega}|\mcal{F}_u \times \mcal{F}_v|dudv\;\; .
\end{eq}

{\noindent Let} $II$ denote the second fundamental form of $\mcal{F}$, and let $K, H$ respectively denote the Gaussian and mean curvatures of $\mcal{F}$.

\subsubsection{First Variation of Area}
Let $\mcal{F} + t\mu$ be a smooth variation of $\mcal{F}$ in the direction of the normal vector $N$. We then have that the derivative of the area with respect  to the variation vector $\mu$ is given by,
\begin{eq}\label{eq:FirstVariationOfArea}
\frac{d\are(\mcal{F} + t\mu)}{dt}|_{t = 0} = \int_{\mcal{F}}\mu H
\end{eq}

\begin{df}\label{MinimalSurface}
	{\em
We will say that $\mcal{F}$ is a} minimal surface {\em in the manifold $\mcal{M}$ if it is a stationary point of the area variation above.}
\end{df}

{\noindent Following} the definition $\mcal{F}$ is minimal if and only if it satisfies,
\begin{eq}\label{eq:MinimalIffZeroMean}
H \equiv 0 \;.
\end{eq}

\subsubsection{Second Variation and Stability}
{\noindent The} term ``minimal'' can be misleading, since, not every minimal surface is a least area surface. Figure \ref{fg:UnStableGeodesic} reflects this in one dimension lower showing curves on a $2$-sphere. The curves, $\alpha$ and $\beta$ shown in the figure, are both geodesic curves. A small variation of $\alpha$ may have shorter length than $\alpha$ itself, while in contrast, any variation of $\beta$ is longer than $\beta$. In such a case we say that $\beta$ is a {\em stable geodesic} and $\alpha$ is {\em unstable}.

\begin{fg}[h]
\[ 
	\setlength{\unitlength}{0.5\standardunitlength}
	\begin{array}{c}  \hspace{-1.7mm}
		\raisebox{-8pt}{\input UnStableGeod.tex }
		\hspace{-1.9mm}
	\end{array}
 \]
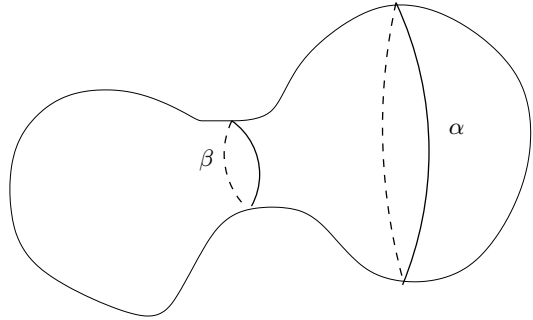
\caption{stable geodesic and non stable geodesic}\label{fg:UnStableGeodesic}
\end{fg}

{\noindent Let} $|II|^2$ denote the squared norm of $II$, and let $Ric_{\mcal{M}}(N)$ denote the {\em Ricci curvature} of $\mcal{M}$ in the direction of $N$. A minimal surface in a $3$-manifold is {\em stable} if it satisfies the following inequality, see  \cite{msy}:

\begin{eq}\label{eq:ShortStabilityCond}
\int_{\mcal{F}}(|II|^2 + Ric_{\mcal{M}}(N)) \leq 0 \;\;,
\end{eq}

In \cite{msy} it is shown that a least area surface is a stable minimal surface.\\

\subsubsection{Some Examples and Existence Results}
A surface $\mcal{F}$ {\em spans} the simple closed curve $\Gamma$, in $\mcal{M}$, if $\del \mcal{F} = \Gamma$.\\ \\
{\noindent \textbf{The case of $\mathbb{R}^3$ }}\\
\begin{enum}
\item Obviously a plane is an example of a stable minimal surface in $\mathbb{R}^3$.
\item If $\mcal{F}:\Omega \rightarrow \mathbb{R}^3$ is a $\mcal{C}^2$ function defined on domain $\Omega \subset \mathbb{R}^2$ that satisfies the minimal surface equation \ref{eq:MinimalIffZeroMean} then it is a least area surface among all surfaces given as functions from $\Omega$ to $\mathbb{R}$, and that coincide with $\mcal{F}$ on $\partial \Omega$.
\end{enum}
{\noindent The} second example above was established independently by Douglas \cite{dou} and Rado \cite{rad}, first for simply connected domains, yielding a least area disk, and was later extended to general domains, for example in \cite{dou2}. \\ \\
\noindent{\textbf{Least Area Surfaces in General $3$-manifolds}}\\
\noindent{A} generalization of the existence of solutions to the Plateau problem for disk type surfaces in general $3$-manifolds was given by Morrey in \cite{morr}. For more general surfaces we have the following:
\begin{thm}\label{thm:LeastAreaMeeksSimonYau}(\cite{msy})
Let $\mcal{M}$ be a compact irreducible $3$-manifold and let $\mcal{F}$ be an embedded incompressible surface in $\mcal{M}$. Then $\mcal{F}$ is isotopic to a least area surface. Moreover, the least area representative in the isotopy class of $\mcal{F}$ is a stable minimal surface.
\end{thm}
{\noindent Using} the stability condition, given in Equation \ref{eq:ShortStabilityCond} Meeks-Simon-Yau deduced the following:
\begin{thm}(\cite{msy})
if $\mcal{M}$ is a compact orientable $3$-manifold with everywhere strictly positive Ricci curvature, and with (possibly empty) boundary of non negative mean curvature (with respect to the outward normal), then $\mcal{M}$ does not contain any compact orientable embedded least area surface of positive genus.
\end{thm}

The following two lemmas will be of crucial importance later as they will be used in Section $3$.
\begin{lem} [\cite{hs}] \label{lem:DiskSpanCurve}
Let $\mcal{M}$ be a closed Riemannian $3$-manifold. There exists $\epsilon > 0$ such that for any point $x \in \mcal{M}$, the ball $B_{\mcal{M}}(x, \; \epsilon)$ satisfies that if $\Gamma \subset \del B_{\mcal{M}}(x, \; \epsilon)$ is a simple closed curve and if $D$ is a least area disk in $\mcal{M}$ spanning $\Gamma$, then $D$ is properly embedded in $B_{\mcal{M}}(x, \; \epsilon)$.
\end{lem}

\begin{lem} [\cite{hs}] \label{lem:ShortEnoughCurve}
Let $\mcal{M}$ be a closed Riemannian manifold. There exists $r > 0$ such that if $\varepsilon < r$ and $\Gamma$ is a closed curve of length less than $\varepsilon$ contained in $B(x, \varepsilon)$, then any least area disk $D$ spanning $\Gamma$ lies in the ball $B(x, \varepsilon)$.
\end{lem}
\noindent{under} appropriate conditions the two lemmas above hold also for compact $3$-manifolds with boundary. This will be referred to in more details in Section $4$. 
\subsubsection{Curvature Bounds}
For a compact $3$-manifold $\mcal{M}$ let $K_{\mcal{M}}$ and $inj_{\mcal{M}}$ respectively denote the maximal sectional curvature and the injectivity radius of $\mcal{M}$. Another result  that will be of extensive use later is the following theorem proved by Schoen.
\begin{thm}(\cite{sch})\label{thm:SchoenBound}
Let $\mcal{F}$ be an immersed stable minimal surface in a compact $3$-manifold $\mcal{M}$, and let $K_{\mcal{F}}$ denote the maximal Gaussian curvature along $\mcal{F}$. Then there exists a constant $C$ that depends only on $K_{\mcal{M}}$ and on $inj_{\mcal{M}}$ so that
\[|K_{\mcal{F}}| < C \;. \]
\end{thm}

\subsection{Normal Surface Theory}\label{sec:PrelimNormal}
 
{\noindent Many} of the classical results in the study of minimal and least area surfaces are based on methods of partial differential equations. These techniques although powerful, sometimes pose difficulties when trying to deduce topological information. Hence, it is desirable to have topological methods and techniques to deal with minimal and least area surfaces. For instance, it may be beneficial to analyse the way least area surfaces intersect with the tetrahedra of some given triangulation of the ambient manifold.\\

{\noindent An} example of such a topological viewpoint is the theory of {\em normal surfaces}. Normal surfaces were first introduced in the 1930`s by Kneser (~\cite{kne}), and further developed by Haken in the 1960`s (~\cite{hak}). This theory studies the ways surfaces intersect with a given triangulation of a $3$-manifold. In this section we will give some preliminary definitions and notations from normal surface theory that will be used later in this work. Most material in this section relies on ~\cite{jr1}, ~\cite{jr2}.

\subsubsection{Normal Surfaces}
Let $\mcal{F}$ be a surface properly embedded in a $3$-manifold $\mcal{M}$, and let $\mcal{T}$ be a given triangulation of $\mcal{M}$.
\begin{df}\label{df:NormalIsotopy}
{\em
An isotopy $H:\mcal{M} \times I: \rightarrow \mcal{M}$ is a} normal isotopy {\em with respect to $\mcal{T}$, if $H_x(\sigma) = \sigma$ for every $x \in I$ and for all simplices $\sigma \in \mcal{T}$. Two submanifolds of $\mcal{M}$ are {\em normally isotopic} if there exists a normal isotopy which carries one to the other.}
\end{df}

\begin{df}\label{df:NormalSurface}
{\em
Let $\mathfrak{t}\in \mcal{T}$ denote a tetrahedron, and let $\mathfrak{f}$ denote a face of $\mathfrak{t}$.
A} normal arc {\em in $\mathfrak{f}$, is an arc properly embedded in $\mathfrak{f}$, such that its end points are on two different edges of $\mathfrak{f}$. A} normal curve {\em in $\mcal{T}^{(2)}$, is a curve properly embedded in $\mcal{T}^{(2)}$, transverse to the edges of $\mcal{T}$, so that its intersection with each $2$-face consists of a possibly empty collection of normal arcs. A} normal disk {\em is a disk $\mcal{D}$, properly embedded in a tetrahedron $\mathfrak{t}$ such that $\partial \mcal{D} = \mcal{D} \cap \partial \mathfrak{t}$, is a simple closed normal curve in $\partial \mathfrak{t}$. An} elementary disk {\em is a normal disk so that its boundary is either a triangle that separates one vertex on $\partial \mathfrak{t}$, from all other three vertices of $\partial \mathfrak{t}$, or a quadrilateral (quad for short) that partitions the vertices of the tetrahedron into two separated pairs.
}
\end{df}
{\noindent Elementary} disks are illustrated in Figure \ref{fg:FigureDiskTypes}, and they come in seven different types corresponding to seven possibilities to partition the set of vertices of a tetrahedron into subsets according to the definition. In the following we will sometimes refer to the boundary of an elementary disk as a $3$-gon or $4$-gon according to its being a triangle or a quad, and for a general normal disk we may refer to its boundary as an $n$-gon.
\begin{df}
{\em
A surface $\mcal{G}$ properly embedded in $\mcal{M}$ is} normal {\em with respect to $\mathcal{T}$ if it is transverse to the skeleta of $\mcal{T}$, and if its intersections with any of the tetrahedra of $\mathcal{T}$ is a (possibly empty) collection of disjoint elementary disks.}
\end{df}
{\noindent Note} that since the surface is embedded, we can have only one type of $4$-gon at each tetrahedron.
\begin{figure}[h]
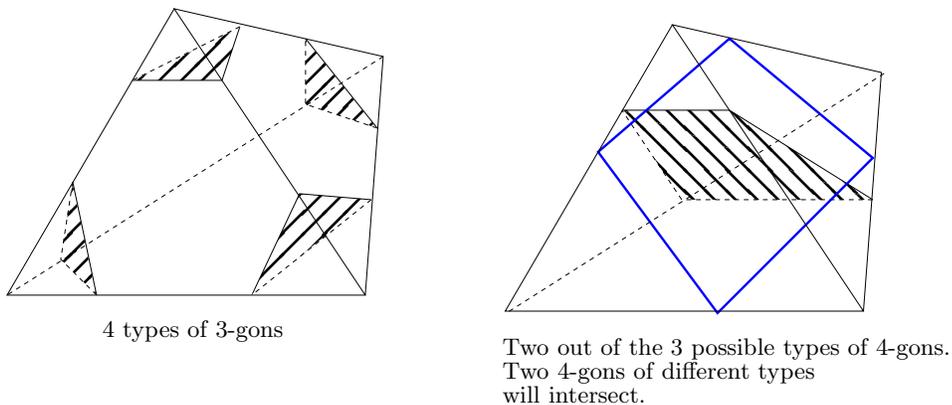

\[
	\setlength{\unitlength}{0.5\standardunitlength}
	\begin{array}{c}  \hspace{-1.7mm}
		\raisebox{-8pt}{\input NormalDisks.tex }
		\hspace{-1.9mm}
	\end{array}
\]
\caption{\small normal disks} \label{fg:FigureDiskTypes}
\end{figure}

\subsubsection{Possible Normal Disk Types}
The {\em length} of a normal curve is the number of normal arcs consisting it.
Recall that the boundary of a normal disk is a simple closed normal curve on the boundary of a tetrahedron that forms a polygon. A thorough analysis of the possible length such a polygon can have, was done by Thompson and Stocking in \cite{tho} and \cite{stoc} respectively. It is summarized as follows.
\begin{lem}(\cite{tho}, \cite{stoc})\label{lem:PossibleGons}
If $\mcal{D}$ is a normal disk then $\partial \mcal{D}$ is an $n$-gon where $n$ is either $3$, $4$ or an even number $\geq 8$.
\end{lem}
{\noindent Figure} \ref{fg:OctagonNormalDisk} shows an example of a normal disk whose boundary is of length eight.
\begin{fg}[H]
\[ 
	\setlength{\unitlength}{0.5\standardunitlength}
	\begin{array}{c}  \hspace{-1.7mm}
		\raisebox{-8pt}{\setlength{\unitlength}{0.00033333in}
\begingroup\makeatletter\ifx\SetFigFont\undefined%
\gdef\SetFigFont#1#2#3#4#5{%
  \reset@font\fontsize{#1}{#2pt}%
  \fontfamily{#3}\fontseries{#4}\fontshape{#5}%
  \selectfont}%
\fi\endgroup%
{\renewcommand{\dashlinestretch}{30}
\begin{picture}(7224,5439)(0,-10)
\thicklines
\dashline{120.000}(1062,3612)(1064,3611)(1069,3609)
	(1079,3606)(1093,3600)(1113,3593)
	(1139,3583)(1170,3571)(1207,3557)
	(1248,3541)(1292,3525)(1339,3507)
	(1387,3489)(1435,3471)(1483,3453)
	(1529,3435)(1575,3418)(1618,3402)
	(1660,3387)(1699,3372)(1737,3359)
	(1772,3346)(1806,3333)(1838,3322)
	(1869,3311)(1899,3300)(1928,3290)
	(1957,3281)(1984,3271)(2012,3262)
	(2044,3252)(2075,3241)(2107,3231)
	(2139,3221)(2171,3211)(2204,3201)
	(2237,3191)(2270,3182)(2303,3172)
	(2336,3163)(2369,3155)(2402,3146)
	(2434,3138)(2466,3131)(2497,3124)
	(2527,3118)(2557,3112)(2585,3106)
	(2613,3102)(2639,3098)(2664,3095)
	(2688,3092)(2711,3090)(2733,3088)
	(2754,3087)(2775,3087)(2797,3087)
	(2820,3089)(2841,3091)(2863,3094)
	(2883,3098)(2904,3103)(2923,3108)
	(2943,3115)(2962,3122)(2980,3131)
	(2998,3140)(3015,3150)(3032,3161)
	(3048,3172)(3063,3183)(3077,3196)
	(3090,3208)(3103,3221)(3115,3234)
	(3127,3247)(3138,3261)(3150,3274)
	(3161,3290)(3173,3307)(3185,3324)
	(3197,3342)(3209,3361)(3221,3381)
	(3232,3401)(3244,3423)(3255,3445)
	(3266,3467)(3276,3490)(3286,3513)
	(3295,3536)(3303,3558)(3311,3581)
	(3318,3603)(3324,3624)(3329,3645)
	(3333,3666)(3337,3687)(3340,3708)
	(3343,3729)(3344,3750)(3346,3772)
	(3346,3794)(3346,3817)(3346,3840)
	(3345,3864)(3343,3888)(3341,3912)
	(3338,3936)(3335,3960)(3331,3984)
	(3328,4007)(3323,4030)(3319,4052)
	(3314,4074)(3310,4095)(3305,4116)
	(3300,4137)(3294,4158)(3289,4179)
	(3283,4200)(3277,4221)(3271,4243)
	(3264,4266)(3257,4289)(3250,4312)
	(3243,4335)(3235,4359)(3228,4382)
	(3220,4404)(3212,4426)(3205,4448)
	(3197,4469)(3190,4489)(3183,4508)
	(3176,4527)(3169,4545)(3162,4562)
	(3154,4581)(3147,4600)(3139,4618)
	(3130,4638)(3121,4658)(3111,4679)
	(3101,4702)(3089,4726)(3077,4752)
	(3065,4778)(3052,4804)(3041,4828)
	(3031,4849)(3022,4866)(3017,4877)
	(3014,4884)(3012,4887)
\path(1062,3612)(912,2037)
\path(3687,4287)(4587,612)
\path(3687,4287)(3012,4887)
\dashline{120.000}(3012,4887)(4737,4437)
\dashline{120.000}(912,2037)(4587,612)
\dashline{120.000}(1587,2712)(5787,987)
\dashline{120.000}(1062,3612)(1587,2712)
\path(4737,4437)(5787,987)
\thinlines
\dashline{60.000}(12,3012)(7212,1512)
\path(3912,5412)(7212,1512)(2787,12)
	(12,3012)(3912,5412)(2787,12)
\end{picture}
} }
		\hspace{-1.9mm}
	\end{array}
\]
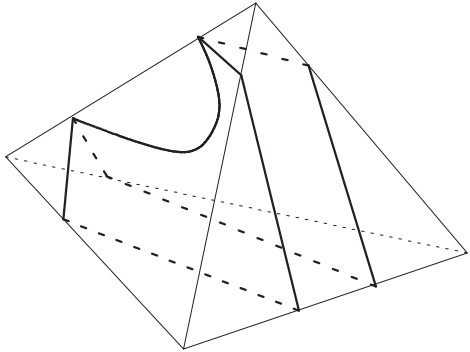
\caption{\small octagon normal disk}\label{fg:OctagonNormalDisk}
\end{fg}

\subsubsection{\it{pl} Minimal Surfaces}
From a bird's-eye view normal surfaces can be considered as a topological version of minimal surfaces. This is illustrated for example by the following theorem that can be viewed as a topological version of Theorem \ref{thm:LeastAreaMeeksSimonYau}

\begin{thm} (\cite{hak} \cite{jr2}) \label{thm:IsotopeSurfaceToNormal}
	Let $\mcal{G}$ be a surface properly embedded in a $3$-manifold $\mcal{M}$. Then, for any triangulation $\mcal{T}$, of $\mcal{M}$, there is a finite sequence of compressions, $\partial$-compressions and local isotopies, that deforms $\mcal{G}$ into a(possibly empty) set of surfaces that are normal, w.r.t. $\mcal{T}$, and another (maybe empty) set of surfaces, each of which is entirely contained in a single tetrahedron of $\mcal{T}$. In particular, if $\mcal{M}$ is irreducible, and $\mcal{G}$ is incompressible and $\partial$-incompressible, then it is isotopic to a normal surface.
\end{thm}

\noindent{In} ~\cite{jr2} a more delicate analogy is made between minimal surfaces and normal surfaces. This is done by introducing $p\ell$-minimal surfaces that are normal surfaces which satisfy some combinatorial minimality constrains. It is shown in ~\cite{jr2} that $p\ell$-minimal surfaces satisfy the following:

\begin{thm}\label{thm:JacoRubinsteinExistence} (\cite{jr2})
In the normal homotopy class of any normal surface  $\mcal{G}$, that is not a {\em vertex linking sphere} (i.e a $2$-sphere which is the link of a vertex), there exists a $p\ell$-minimal surface. 
\end{thm}

{\noindent In} \cite{ni} an altered definition of $p\ell$ minimal surfaces is given, and in respect to this definition the following is proved.
\begin{thm}[\cite{ni}]
There is exactly one $p\ell$-minimal surface in a normal homotopy class of any normal surface $\mcal{G}$ that is not a vertex linking sphere.
\end{thm}

\subsection{Fat triangulations }\label{sec:PrelimFat}
In this section we will give basic definitions notations and quote relevant results about fat triangulations. All details can be found in the relevant cited literature.

\subsubsection{Fat Triangulations}
\begin{df}
\begin{enumerate}
\item {\em A triangle in $\mathbb{R}^2$ is called} $\varphi$-fat {\em iff  all its angles are larger than a prescribed value $\varphi> 0$.}
\item {\em A k-simplex $\tau \subset \mathbb{R}^n$, $2 \leq k \leq n$, is $\varphi${\em -fat} if there exists $\varphi > 0$ such that the ratio $\frac{r}{R} \geq \varphi$, where $r$ and $R$, are resp. the radii of the inscribed and circumscribed (k-1)-spheres of $\tau$.}
\item {\em A triangulation $\mathcal{T} = \{ \sigma_i \}_{i\in \bf I }$ is {\it fat} if all its simplices are $\varphi$-{\em fat} for some $\varphi > 0$.}
\end{enumerate}
\end{df}

{\noindent The} importance of fat triangulations is stressed by the following example.
\subsubsection{Schwartz Lantern} The Schwartz lantern is given by the following:
\begin{exm}\label{exm:Lantern}
{\em Let $\mcal{G}$ be the standard cylinder $\mathbb{S}^1 \times I$ of hight $1$ and radius $1$ in $\mathbb{R}^3$. Let $\mcal{G}_n$ be the ``lantern'' that is given as the $pl$ triangulated surface so that the vertices of its triangulation are placed on $\mcal{G}$. Suppose the set of vertices of $\mcal{G}_n$ consists of $n^4$ points so that $n$ points are evenly placed along the circle $\mathbb{S}^1 \times \{1\}$ and $n^3$ points are placed along the vertical direction as indicated in the following Figure \ref{fg:Lantern}.
\begin{fg}[H]
\[ 
	\setlength{\unitlength}{0.5\standardunitlength}
	\begin{array}{c}  \hspace{-1.7mm}
		\raisebox{-8pt}{\begin{picture}(0,0)%
\includegraphics{Lantern.pstex}%
\end{picture}%
\setlength{\unitlength}{2368sp}%
\begingroup\makeatletter\ifx\SetFigFont\undefined%
\gdef\SetFigFont#1#2#3#4#5{%
  \reset@font\fontsize{#1}{#2pt}%
  \fontfamily{#3}\fontseries{#4}\fontshape{#5}%
  \selectfont}%
\fi\endgroup%
\begin{picture}(6075,6525)(1951,-7336)
\end{picture}%
 }
		\hspace{-1.9mm}
	\end{array}
 \]
\caption{Schwartz Lantern}\label{fg:Lantern}
\end{fg}
{\noindent Simple} computation shows that each of the $pl$ triangles of $\mcal{G}_n$ has base length $\geq \frac{1}{2n}$ and hight $\geq \frac{1}{4n}\arctan(\frac{\pi}{2n})$. Following this, the area of the lantern thus satisfies:
\begin{eq}\label{eq:LanternBlowsI}
\are(\mcal{G}_n) \geq n^4\frac{1}{2n}\frac{1}{4n}\arctan(\frac{\pi}{2n}) \;.
\end{eq}

{\noindent As} $n \rightarrow \infty$, $\arctan(\frac{\pi}{2n}) \rightarrow \frac{\pi}{2n} > \frac{1}{2n}$, hence the following holds:
\begin{eq}\label{eq:LanternBlowsII}
\are(\mcal{G}_n) \geq \frac{n}{8} \;,
\end{eq}

{\noindent so} $\are(\mcal{G}_n)$ blows up to $\infty$ as $n$ goes to $\infty$. On the other hand, it can be shown that if the $n^4$ points are places so that the triangles of $\mcal{G}_n$ are fat then $\are(\mcal{G}_n) \rightarrow \are(\mcal{G})$.
}
\end{exm}

\begin{thm}(\cite{cms})
There exists a constant $c(k)$ that depends solely upon the
dimension $k$ of  $\tau$ such that
\begin{equation} \label{eq:ang-cond}
\frac{1}{c(k)}\cdot \varphi(\tau) \leq \min_{\hspace{0.1cm}\sigma
< \tau}\measuredangle(\tau,\sigma) \leq c(k)\cdot \varphi(\tau)\,,
\end{equation}
{\noindent and}
\begin{equation} \label{eq:area-cond}
\varphi(\tau) \leq \frac{Vol_j(\sigma)}{diam^{j}\,\sigma} \leq c(k)\cdot \varphi(\tau)\,,
\end{equation}
{\noindent where} $\varphi$ denotes the fatness of the simplex $\tau$, $\measuredangle(\tau,\sigma)$ denotes the  ({\em
internal}) {\em dihedral angle} of the face $\sigma < \tau$ and $Vol_{j}(\sigma)$; $diam\,\sigma$ stand for the
Euclidian $j$-volume and the diameter of $\sigma$ respectively. (If $dim\,\sigma = 0$, then $Vol_{j}(\sigma) = 1$,
by convention.)
\end{thm}
{\noindent Condition}  \ref{eq:ang-cond} is just the expression of fatness as
a function of dihedral angles in all dimensions, while Condition
\ref{eq:area-cond} expresses fatness as given by ``large
area/diameter''.

\subsubsection{Existence Results}
Existence of fat triangulations of Riemannian manifolds is guaranteed by the studies given in \cite{ca}, \cite{pel}, \cite{sa1}, \cite{sa2} and \cite{bre}. These are summarized below.

\begin{thm}(\cite{ca})\label{thm:cairns}
Every compact $\mathcal{C}^2$ Riemannian manifold  admits a fat triangulation.
\end{thm}

\begin{thm}(\cite{pel})\label{thm:peltonen}
Every open (unbounded) $\mathcal{C}^\infty$ Riemannian manifold admits a fat triangulation.
\end{thm}

\begin{thm}(\cite{sa2})\label{thm:ext-to-int}
Let $M^{n}$ be an $n$-dimensional $\mathcal{C}^{1}$ Riemannian manifold with boundary, having a finite number of
compact boundary components. Then, any  fat triangulation of $\partial M^{n}$ can be extended to a fat
triangulation of $M^{n}$.
\end{thm}
{\noindent Following} a thick-thin decomposition theorem for hyperbolic manifolds, W. Breslin (\cite{bre}) uses the term {\em thick triangulation} instead of fat.
\begin{thm}(\cite{bre})
Let $\mu$ be a Margulis constant for $\mathbb{H}^n, \; n\geq 2$. There exists a constant $L = L(n, \mu)$ such that every complete hyperbolic $n$-manifold has a $(\mu, L)$ thick triangulation.
\end{thm}

\subsubsection{Fat Traingulations and Manifolds Approximations}
\begin{df}
	{\em Let $(\mcal{M}, \mcal{T})$ be an $m$-dimensional triangulated manifold, and let $f$ be a $\mathcal{C}^r$ embedding of $(\mcal{M}, \mcal{T})$ into $\mathbb{R}^n$. Let $\delta$ be a positive real number. Let $|K| \subset \mathbb{R}^n$ be an embedding of an $m$-simplicial complex $K$. Denote by $d_{eucl}$ the euclidean distance in $\mathbb{R}^n$, and by $St(a, K)$ the star of a point $a$ in the complex $K$. Then $g:|K| \rightarrow \mathbb{R}^n$ is called a} $\delta$-approximation {\em to $f$ if:
		\begin{enum}
			\item There exists a subdivision $K'$ of $K$ such that $g \in \mathcal{C}^r(K',\mathbb{R}^n)$
			\item $d_{eucl}\big(f(x),g(x)\big) < \delta(x)$, for any $x \in |K|$
			\item $d_{eucl}\big(df_a(x),dg_a(x)\big) \leq \delta(a)\cdot d_{eucl}(x,a)$\,, for any $a \in |K|$ and for all $x \in \overline{St}(a,K')$.
		\end{enum}
	}
\end{df}
\begin{df}\label{df:SecantMapII}
	{\em
		Let $(\mcal{M}, \mcal{T})$ be a $\mcal{C}^r$ triangulated $m$-manifold embedded in $\mathbb{R}^n$. Let $s \in \mcal{T}$ be a simplex. The {\em vertex map} $V_s:s \rightarrow
		\mathbb{R}^n$, is determined by $V_s(v) = v$, where $v$ is a vertex of $s$. The linear extension of $V_s$ is called the} secant map induced by
	$\mcal{T}$, {\em and it is denoted by $L_s$}
\end{df}
{\noindent The} motivation for having fat triangulations for manifolds in terms of $p\ell$-approximations is stressed by the following theorem.

\begin{thm}\label{thm:Munkres1}(\cite{mun})
	Let $(\mcal{M}, \mcal{T})$ be a $\mathcal{C}^r$ triangulated $m$-manifold embedded in $\mathbb{R}^n$. Then, for $\delta > 0$, there exists
	$\epsilon, \varphi_0 >0$, such that, if any $\tau < \mcal{T}$, fulfills the following conditions:
	\begin{enum}
		\item $diam(\tau)\; < \; \epsilon$
		\item $\tau$ is $\varphi_0$-fat
	\end{enum}
	then, the secant map $L_\tau$ is a $\delta$-approximation of $f|\tau$.
\end{thm}

\begin{df}\label{df:HausdorffMetric}
	{\em
		Let $\mcal{M}$ be a Riemannian manifold. The} Hausdorff metric {\em is a metric defined on the set of compact subsets of $\mcal{M}$ by setting:}\\
	$d_H(A, B) = \max_{a \in A}\{d_{\mcal{M}}(a, B)\} + \max_{b \in B}\{d_{\mcal{M}}(b, A\}$, {\em for any two compact subsets $A$ and $B$, of $\mcal{M}$, see \cite{mor}.
	}
\end{df}

{\noindent We} will make a use of the following lemma from \cite{mor}
\begin{lem}\label{lem:MorvanAreaAppx}
	Let $\mcal{S}$ be a surface embedded in the compact triangulated $3$-manifold $(\mcal{M}, \mcal{T})$ and let $\mcal{S}_j$ be a sequence of surfaces embedded in $\mcal{M}$ that converges to $\mcal{S}$ in the Hausdorff metric on the set of embedded compact surfaces in $\mcal{M}$. Suppose that in addition to the metric convergence of $\mcal{S}_j$ to $\mcal{S}$ the following holds:
	For every $x \in \mcal{S}$ and $x_j \in \mcal{S}_j$ so that $x = \lim_{j \rightarrow \infty} x_j$,
	\begin{eq}
		\lim_{j \rightarrow \infty} N_{x_j}\mcal{S}_j = N_{x}\mcal{S} \;,
	\end{eq}
	where $N_{x_j}\mcal{S}_j, N_{x}\mcal{S}$ are the unit normals of $\mcal{S}$ at $x$ and $\mcal{S}_j$ at $x_j$ respectively. Then
	\begin{eq}
		\lim_{j \rightarrow \infty} \are(\mcal{S}_j) = \are(\mcal{S})
	\end{eq}
\end{lem}

{\noindent As} the secant approximation admits a natural Euclidean structure, Theorem \ref{thm:Munkres1} can be used in order to show relations between the geometric structure of the original manifold $(\mcal{M}, \mathfrak{g})$ and its Euclidean secant approximation. Suppose $(\mcal{M}, \mathfrak{g})$ is a Riemannian manifold with constant sectional curvature $K$. Let $(\mcal{M}_t, \mathfrak{g}_t)$ denote the manifold obtained from $(\mcal{M}, \mathfrak{g})$ by scaling the metric $\mathfrak{g}$ by a real scaling factor $t > 0$.That is, the scaled metric $\mathfrak{g}_t$ is given by $\mathfrak{g}_t = \frac{1}{t}\mathfrak{g}$. Then  $(\mcal{M}_t, \mathfrak{g}_t)$ is a Riemannian manifold with constant sectional curvature $K_t = t^2K$, and the sequence of manifolds $(\mcal{M}_t, \mathfrak{g}_t)$ converges to a Euclidean manifold $\widehat{\mcal{M}}$, (see \cite{por}, pp. 369).


\section{Quasi-Normality of Least Area Surfaces}
In this section we examine the way in which a least area surface meets a triangulation of a $3$-manifold. It will be shown that under some considerations such a surface meets the triangulation in a fairly simple way.

\subsection{Notations and Definitions}
Throughout the rest of this section $(\mcal{M}, \mcal{T})$ will denote a smooth triangulated $3$-manifold. It will be assumed that $\mcal{M}$ admits either Euclidean or hyperbolic Reimannian structure. Recall the following notations: The injectivity radius of $\mcal{M}$ is denoted by $inj_{\mcal{M}}$, the maximal sectional curvature of $\mcal{M}$ by $K_{\mcal{M}}$, and the mesh of a triangulation $\mcal{T}$ of $\mcal{M}$, by $\lambda(\mcal{T})$.
\begin{df}
{\em A fixed surface $\mcal{F}$ embedded in the triangulated $3$-manifold $(\mcal{M}, \mcal{T})$, is} quasi-normal {\em with respect to $\mcal{T}$ if $\mcal{F} \cap \mcal{T}^{(3)}$ is a collection of disks, such that every normal disk in this collection is an elementary disk, and for every $2$-face $\mathfrak{f}$, $\mcal{F} \cap int(\mathfrak{f})$ does not contain any simple closed curve}.
\end{df}
{\noindent Note} that the difference between quasi-normal surface and a normal one is that the intersection of a quasi-normal and the $2$-skeleton of the triangulation may contain curves having both their end points on the same edge of $\mcal{T}^{(1)}$. Following Schleimer, \cite{schl}, such a curve will be called a {\em bent curve}. It is illustrated in Figure \ref{fg:BentCurve} below.

\begin{fg}[H]
	\[ 
	\setlength{\unitlength}{0.5\standardunitlength}
	\begin{array}{c}  \hspace{-1.7mm}
		\raisebox{-8pt}{\begin{picture}(0,0)%
\includegraphics{BentCurve.pstex}%
\end{picture}%
\setlength{\unitlength}{1973sp}%
\begingroup\makeatletter\ifx\SetFigFont\undefined%
\gdef\SetFigFont#1#2#3#4#5{%
  \reset@font\fontsize{#1}{#2pt}%
  \fontfamily{#3}\fontseries{#4}\fontshape{#5}%
  \selectfont}%
\fi\endgroup%
\begin{picture}(5274,3324)(3439,-4423)
\put(6001,-3136){\makebox(0,0)[lb]{\smash{{\SetFigFont{9}{10.8}{\rmdefault}{\mddefault}{\updefault}{\color[rgb]{0,0,0}$\gamma$}%
}}}}
\end{picture}%
 }
		\hspace{-1.9mm}
	\end{array}
 \]
	\caption{a bent curve}\label{fg:BentCurve}
\end{fg}

\renewcommand{\thethm}{\arabic{thm}}
\setcounter{thm}{0}
\begin{thm} \label{thm:LeastAreaIsNormal}
Let $\mcal{M}$ be a closed, irreducible, orientable, Euclidean or hyperbolic $3$-manifold. Let $\mcal{F}$ be a fixed closed, two-sided, incompressible, least area surface in $\mcal{M}$. Let $\mcal{T}$ be some geodesic least area triangulation of $\mcal{M}$ so that $\mcal{F} \cap \mcal{T}^{(0)} = \emptyset$. Then for every $0 < \varphi < \frac{\pi}{2}$ there exists a constant $\Psi > 0$ that depends only on $inj_{\mcal{M}}, K_{\mcal{M}}$, and $\varphi$, such that if $\mcal{T}$ is $\varphi$-fat, and $\lambda(\mcal{T}) < \Psi$, then $\mcal{F}$ is quasi-normal with respect to $\mcal{T}$.
\end{thm}
\renewcommand{\thethm}{\thesection.\arabic{thm}}
\setcounter{thm}{1}

\begin{rem}
{\em
Recall that according to Theorem \ref{thm:IsotopeSurfaceToNormal} (see $2.4$ of \cite{jr1}), any incompressible surface can be isotoped to a surface that is normal with respect to any given triangulation $\mcal{T}$. In this context Theorem \ref{thm:LeastAreaIsNormal} states that if the surface is also least area and the triangulation is fat with small enough mesh, then the surface is already in a position that, in some sense, is close to normal position.
}
\end{rem}

{\noindent We} will make use of the following claim:
\begin{clm}\label{clm:EnoughForScaled}
It is enough to prove Theorem \ref{thm:LeastAreaIsNormal} for a scaled manifold $(\mcal{M}_t, \mathfrak{g}_t)$, where $\mathfrak{g}_t = \frac{1}{t}\mathfrak{g}$ is a scaled version of the metric $\mathfrak{g}$.
\end{clm}
\begin{proof}
According to ~\cite{por} the scaled manifold $\mcal{M}_t$ is homeomorphic to the original manifold $\mcal{M}$, and has sectional curvature $K_t = t^2K$. In addition, since the scaling factor $\frac{1}{t}$ is positive we have that a least area surface is mapped isotopically onto a least area surface with respect to the scaled metric $\mathfrak{g}_t$ inside $\mcal{M}_t$. Moreover, a $\varphi$-fat geodesic least area triangulation $\mcal{T}$ of $\mcal{M}$ is mapped onto a $\varphi$-fat geodesic least area triangulation $\mcal{T}_t$ of $\mcal{M}_t$. As a result, if $\mcal{F}_t$ is quasi normal with respect to $\mcal{T}_t$ then $\mcal{F}$ must be quasi normal with respect to $\mcal{T}$, as the intersection patterns do not change along the scaling$\setminus$rescaling process. This proves the claim.
\end{proof}

Following Claim \ref{clm:EnoughForScaled} and the fact that $K_t \rightarrow 0$ as $t \rightarrow 0$, we have that the Schoen curvature bound $C_t$, given in \ref{thm:SchoenBound} also tends to zero. Hence, without loss of generality we will assume that the curvature bound $C$ is smaller than $2\varphi$. This fact will be of major significance for the proof of Theorem \ref{thm:LeastAreaIsNormal}.\\

\subsection{Essential Lemmas}
{\noindent Theorem} \ref{thm:LeastAreaIsNormal} will be deduced from the following three Lemmas.

\begin{lem}\label{lem:NoSimpleClosedCurves}
There exists $\epsilon > 0$ such that if $\mcal{T}$ is $\varphi$-fat for some positive $\varphi$, and $\lambda(\mcal{T}) < \epsilon$ then the intersection of $\mcal{F}$ with the interior of any of the $2$-faces of $\mcal{T}$ does not contain simple closed curves.
\end{lem}

\begin{proof} Suppose conversely that $\gamma$ is a simple closed curve in the interior of some $2$-face $\mathfrak{f}$. Then $\gamma$ bounds a disk in $\mathfrak{f}$, denoted by $\mcal{E}$, and since $\mcal{F}$ is incompressible $\gamma$ must also bound a disk $\mcal{D} \subset \mcal{F}$, for otherwise, $\mcal{E}$ will be a compressing disk for $\mcal{F}$.\\
{\noindent {\bf Case I:}} Assume that the disk $\mcal{D}$ is an innermost disk in $\mcal{F}$.
\begin{fg}[H]
\[ 
	\setlength{\unitlength}{0.5\standardunitlength}
	\begin{array}{c}  \hspace{-1.7mm}
		\raisebox{-8pt}{\input NonNormalCloseCurve.tex }
		\hspace{-1.9mm}
	\end{array}
 \]
\caption{\small $\gamma$ bounds a disk $\mcal{D} \subset \mcal{F}$ }
\end{fg}
{\noindent Therefore} the interior of $\mcal{D}$ is entirely contained in some tetrahedron $\tau$. In this case $\mcal{S} = \mcal{E} \cup \mcal{D}$ is a $2$-sphere embedded in $\mcal{M}$, and since $\mcal{M}$ is irreducible, $\mcal{S}$ bounds a $3$-ball $\mcal{B}$ in $\mcal{M}$, through which $\mcal{D}$ can be isotoped to $\mcal{E}$ without changing $\mcal{F}$ outside of $\mcal{B}$. The resulting surface is denoted by $\mcal{F}'$, and it is isotopic to $\mcal{F}$. Moreover, it coincides with $\mcal{F}$ outside of $\mcal{B}$. Suppose this isotopy does not reduce the area of $\mcal{F}$. Let $\mathfrak{f}' = (\mathfrak{f} \setminus \mcal{E})\cup \mcal{D}$, that is, $\mathfrak{f}'$ is a disk obtained from the $2$-face $\mathfrak{f}$ by gluing $\mcal{D}$ instead of $\mcal{E}$ along their common boundary. Note that both disks $\mathfrak{f}$ and $\mathfrak{f}'$ span $\partial \mathfrak{f}$. According to the assumption one has that $\are(\mathfrak{f}') \leq \are(\mathfrak{f})$, hence since $\mathfrak{f}$ is a least area disk $\mathfrak{f}'$ is also a least area disk. Note that the following also holds:
\begin{eq}\label{eq:NonEmptyInteriors}
int(\mathfrak{f}') \cap int(\mathfrak{f}) = int(\mathfrak{f}) \setminus \mcal{E} \neq \emptyset .
\end{eq}

{\noindent According} to Theorem $8$ of \cite{my3} if $\mcal{M}$ is a $3$-manifold whose boundary is piecewise smooth and the mean curvature of the boundary is non negative, then any two least area disks that span the same simple close curve on the boundary of $\mcal{M}$, have either distinct or identical interiors. In our case the corresponding $3$- manifold is the tetrahedron $\tau$, which has piecewise smooth boundary with non negative mean curvature, since every $2$-face is a least area disk. The simple close curve on the boundary of $\mcal{M}$ is the boundary of the $2$-face $\mathfrak{f}$. Therefore, since $\mathfrak{f}$ and $\mathfrak{f}'$ are least area disks having the same boundary curve, their interiors either coincide or distinct. This is in contradiction to Equation \ref{eq:NonEmptyInteriors} above. Hence, the isotopy from $\mcal{F}$ onto $\mcal{F}'$ does reduce area, which is in contradiction to the assumption that $\mcal{F}$ is a least area surface. This completes the proof of the lemma in case I.\\

{\noindent {\bf Case II:}} Suppose now that within the collection of intersection curves of $\mcal{F}$ with the interior of $2$-faces, such as the curve $\gamma$, there is no curve for which the corresponding disk $\mcal{D}$ is innermost in $\mcal{F}$. In this case let $\gamma$ be one of these curves so that the disk $\mcal{E}$ that $\gamma$ bounds inside the $2$-face $\mathfrak{f}$ is innermost in $\mathfrak{f}$. Such $\gamma$ always exists due to compactness. In particular, it follows from the assumption that $int(\mcal{D})$ is not entirely contained in a single tetrahedron. The intersection of $int(\mcal{D})$ with the $2$-skeleton of $\mcal{T}$ is a collection of simple closed curves each is the intersection of $int(\mcal{D})$ with the boundary of some tetrahedron. Again, due to compactness there exists a curve $\gamma$ for which the number of simple closed curve in $\mcal{D} \cap \mcal{T}^{(2)}$ is minimal. Then there exists a tetrahedron $\tau$ such that the curve $\gamma$ is on a $2$-face of $\tau$ and such that the disk $\mcal{D}$ intersects $\tau$ in at least one additional curve $\beta$. Since $\partial \tau$ is a $2$-sphere the curve $\beta$ also bounds a disk in $\partial \tau$, denoted by $\mcal{E}'$, and from the minimality condition on $\gamma$ it follows that $\mcal{E}'$ cannot be entirely contained in the interior of a $2$-face. Hence, the curve $\beta$ intersects the set of edges of $\tau$. Let $v$ be any vertex of $\tau$, and let $\theta_1, \theta_2, \theta_3$ denote the three dihedral angles around $v$ (i.e. along each of the three edges of $\tau$ meeting at $v$). Since $\mcal{T}$ is $\varphi$-fat, each of these angles is at least $\varphi$, and since the manifold $\mcal{M}$ is non positively curved the sum of these angles does not exceed $\pi$ putting this together imposes that $\theta_1 + 2 \cdot \varphi \leq \theta_1 + \theta_2 + \theta_3 \leq \pi$, from which it follows that $\theta_1 \leq \pi - 2 \cdot \varphi$. The same holds for the other two angles and in fact to all dihedral angles of the triangulation $\mcal{T}$. From Theorem $1.1$ of \cite{hr} it follows that the disk $\mcal{E}'$ is a strictly convex region of $\partial \tau$. According to Theorem $1$ of \cite{my3} the curve $\beta$ bounds a least area disk $\mcal{D}''$ properly embedded inside $\tau$, and since $\mcal{F}$ is incompressible it also bounds a disk $\mcal{D}'$ in $\mcal{F}$ which, according to the assumption is not contained in $\tau$. Moreover, according to Lemma \ref{lem:DiskSpanCurve} there exists $\epsilon > 0$ so that if $\lambda(\mcal{T})$ is smaller than $\epsilon$ then every least area disk that spans $\beta$ is properly embedded inside $\tau$. in particular this means that the disk $\mcal{D}'$ cannot be a least area disk. Let $\mcal{S} = \mcal{D}'' \cup \mcal{D}'$, then $\mcal{S}$ is a $2$-sphere in $\mcal{M}$ which bounds a $3$-ball due to the irreducibility of $\mcal{M}$. If $\mcal{S}$ is embedded in $\mcal{M}$ it gives rise to an isotopy that takes $\mcal{D}'$ to $\mcal{D}''$ thus isotoping $\mcal{F}$ to a surface of smaller area forming a contradiction.\\

{\noindent In} order to complete the proof, it is still necessary to consider the case in which $\mcal{S}$ is not an embedded sphere. This can happen if the disks $\mcal{D}'$ and $\mcal{D}''$ intersect in their interiors. Since $\mcal{D}'$ is not a least area disk, it can be deformed slightly, if necessary, so to ensure that $\mcal{D}'$ and $\mcal{D}''$ intersect transversely, so every component of the intersection is a simple closed curve. Suppose $\mcal{D}'$ and $\mcal{D}''$ are such that the number of components in their intersection is minimal among all such pairs. Let $\delta \in \mcal{D}' \cap \mcal{D}''$ be an innermost curve in $\mcal{D}''$, and let $\mcal{E}''$ and $\mcal{E}'''$ be the disks bounded by $\delta$ on $\mcal{D}'$ and $\mcal{D}''$ respectively. Then $int(\mcal{E}'') \cap int(\mcal{E}''') = \emptyset$, hence $\mcal{S}' = \mcal{E}'' \cup \mcal{E}'''$ is an embedded $2$-sphere. Note that $\delta$ may not be innermost on $\mcal{E}'''$, since the disk $\mcal{D}'$ may still intersect $int(\mcal{E}''')$ outside $\mcal{E}''$, as illustrated in Figure \ref{fig:NonDiskInductionStep}.

\begin{fg}[H]
\[ 
	\setlength{\unitlength}{0.5\standardunitlength}
	\begin{array}{c}  \hspace{-1.7mm}
		\raisebox{-8pt}{\input NonDiscTwoCompInductionI.tex }
		\hspace{-1.9mm}
	\end{array}
 \]
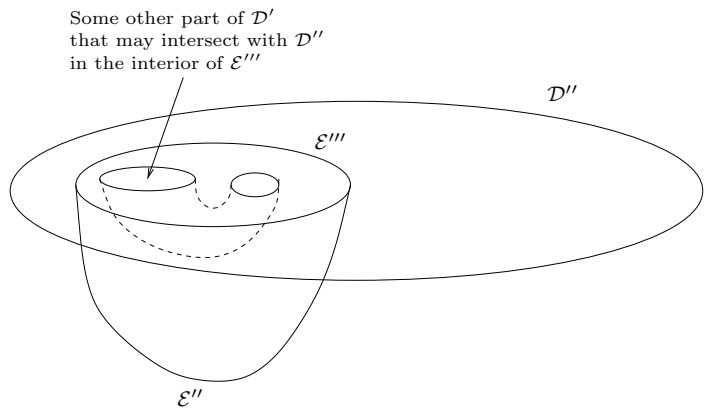
\caption{looking for an embedded $2$-sphere}\label{fig:NonDiskInductionStep}
\end{fg}

{\noindent Isotoping} $\mcal{E}''$ through $\mcal{E}'''$ eliminates $\delta$ along with all other possible intersections of $\mcal{D}'$ with the interior of $\mcal{E}'''$, thus reducing the number of connected components by at least $1$. This contradicts the assumption that $\mcal{D}'$ and $\mcal{D}''$ minimize the number of intersections. This completes the proof of the lemma.
\end{proof}

{\noindent The} next two lemmas are concerned with the intersection of a least area surface with the interiors of tetrahedra. First, it is necessary to show that this intersection is made of a collection of disks. In general, this may not hold (see \cite{cm1}--~\cite{cm5}). However, if the mesh of the triangulation is considered, then the following holds:

\begin{lem}\label{lem:IntesectionsAreDisks}
There exists $\epsilon > 0$ such that if $\lambda(\mcal{T}) < \epsilon$, then the intersection of $\mcal{F}$ with any tetrahedron of $\mcal{T}$ is a (possibly empty) collection of disks.
\end{lem}

\begin{proof}
Let $\mcal{G} = \mcal{F} \cap \tau$, for some tetrahedron $\tau \in \mcal{T}^{(3)}$, and let $\mcal{S}$ be some connected component of $\mcal{G}$, properly embedded in $\tau$.
Suppose conversely, that $\mcal{S}$ is not a disk. Let $\{\alpha_1, \alpha_2, ..., \alpha_n \} = \partial \mcal{S} = (\mcal{S} \cap \partial \tau)$. See Figure \ref{fg:NonDiskIntersectionI} for an illustration in the case $n = 3$.
\begin{fg}[H]
\[ 
	\setlength{\unitlength}{0.5\standardunitlength}
	\begin{array}{c}  \hspace{-1.7mm}
		\raisebox{-8pt}{\input NonDiskTwoComponents.tex }
		\hspace{-1.9mm}
	\end{array}
 \]
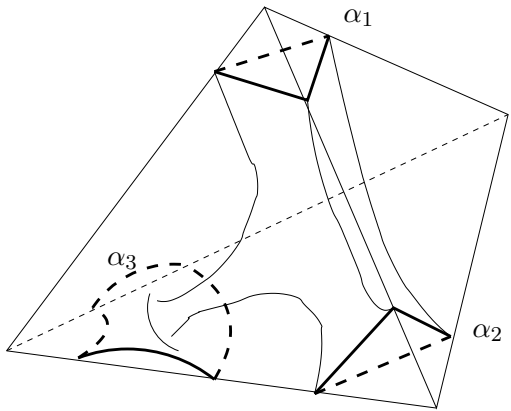
\caption{non-disk intersection}\label{fg:NonDiskIntersectionI}
\end{fg}
{\noindent Since} $\partial \tau$ is a $2$-sphere, each of the curves $\alpha_j$ bounds a disk in $\partial \tau$, and since $\mcal{F}$ is incompressible each $\alpha_j$ must also bound a disk, $\mcal{D}_j$ in $\mcal{F}$. If all disks $\{\mcal{D}_j \; ; j = 1, ..., n \}$ are contained in the tetrahedron $\tau$, then $\mcal{S}$ is a $2$-sphere. In this case $\mcal{S}$ can be shrunk to a point inside $\tau$, thus reducing the area of $\mcal{F}$. Therefore, there exists some $1 \leq i \leq n$ for which $\mcal{D}_i$ is not entirely contained in $\tau$ (See Figure \ref{fg:DiskFromOutside} for an illustration).
\begin{fg}[H]
\[ 
	\setlength{\unitlength}{0.5\standardunitlength}
	\begin{array}{c}  \hspace{-1.7mm}
		\raisebox{-8pt}{\input NonDiskTwoCompOutDisk.tex }
		\hspace{-1.9mm}
	\end{array}
 \]
\caption{$\alpha$ bounds a disk outside of $\tau$}\label{fg:DiskFromOutside}
\end{fg}
{\noindent Since} $\alpha_i = \partial \mcal{D}_i$ is embedded in $\partial \tau$, and since due to Lemma \ref{lem:NoSimpleClosedCurves} $\alpha_i$ cannot be entirely embedded in a $2$-face of $\tau$, Theorem $1.1$ of \cite{hr} is used again from which it follows that $\alpha_i$ bounds a strictly convex disk $\mcal{E} \subset \partial \tau$. Again, according to Meeks-Yau (\cite{my3}, Theorems $1$) and to Lemma \ref{lem:DiskSpanCurve}, there exists $\epsilon > 0$ so that if $\lambda(\mcal{T}) < \epsilon$ then $\alpha_i$ bounds a least area disk $\mcal{D}''$ that is properly embedded in the tetrahedron $\tau$, and $\mcal{D}_i$ cannot be a least area disk. Let $\mcal{S}'$ be the $2$-sphere $\mcal{D}_i \cup \mcal{D}''$. As in Lemma \ref{lem:NoSimpleClosedCurves} if $\mcal{S}'$ is embedded $\mcal{F}$ can be isotoped to a surface $\mcal{F}'$ of smaller area, in contradiction to the choice of $\mcal{F}$.\\
{\noindent If} $\mcal{S}'$ is not embedded then the procedure used in Lemma \ref{lem:NoSimpleClosedCurves} is repeated. $\mcal{D}^i$ and $\mcal{D}''$ are chosen such that the number of components in their intersection is minimal amongst all such pairs. Let $\gamma \in \mcal{D}_i \cap \mcal{D}''$ be an innermost curve in $\mcal{D}_i$, and let $\mcal{E}$ and $\mcal{E}''$ be the disks bounded by $\gamma$ on $\mcal{D}_i$ and $\mcal{D}''$ respectively. Then $int(\mcal{E}) \cap int(\mcal{E}'') = \emptyset$, hence $\mcal{S}'' = \mcal{E} \cup \mcal{E}''$ is an embedded $2$-sphere. Isotoping $\mcal{E}$ through $\mcal{E}''$ eliminates $\gamma$ thus reducing the number of connected components by at least $1$. This contradicts the assumption that $\mcal{D}_i$ and $\mcal{D}''$ minimize the number of intersections. The proof of the lemma is completed.
\end{proof}

{\noindent Recall} that a disk $\mcal{D}$ embedded in some tetrahedron $\tau$ is normal if its boundary is a collection of normal arcs in $\partial \tau$. Recall also that a normal disk is elementary if its boundary is made of either $3$ or $4$ normal arcs. By Thompson, (\cite{tho}) and Stocking (\cite{stoc}) the  following holds:

\begin{lem}\label{thm:ThompsonStocking}(\cite{tho} pp. 615; \cite{stoc} Claim $1$, Claim $2$)
All possible normal disks are $3$-gons, $4$-gons and $n$-gons, where $n$ is even and greater or equal to $8$.
\end{lem}

{\noindent In} the following lemma it is proved that if $\mcal{T}$ is a fat triangulation with small enough mesh, then every normal disk is an elementary disk.

\begin{lem}\label{lem:3Or4Gon}
For every $0 < \varphi < \frac{\pi}{2}$ there exists $\epsilon > 0$ that depends only on $inj_{\mcal{M}}, K_{\mcal{M}}$ and $\varphi$, such that if $\lambda(\mcal{T}) < \epsilon$, and $\mcal{T}$ is $\varphi$-fat, then every normal disk in $\mcal{F} \cap \mcal{T}$ is either a $3$-gon or a $4$-gon.
\end{lem}

\begin{proof}
Suppose there is some tetrahedron $\tau \in \mcal{T}$ for which $\mcal{F} \cap \tau$ contains some connected component $\mcal{D}$ with at least $8$ edges. Using \cite{tho} and elementary combinatorial argument, one can easily show that there is at least one pair of opposite edges of $\tau$ (i.e., edges that do not share a vertex), so that $\mcal{D}$ intersects each of these edges at least twice. Let $p$ and $q$ be a pair of intersection points on one of these edges, and let $x$ and $y$ be a pair of intersection points on the opposite edge. See Figures \ref{fig:OctagonDisk1} for an illustration in the case where $\partial \mcal{D}$ is made of eight edges.
\begin{fg}[H]
\[ 
	\setlength{\unitlength}{0.5\standardunitlength}
	\begin{array}{c}  \hspace{-1.7mm}
		\raisebox{-8pt}{\input NormalOctagonA.tex }
		\hspace{-1.9mm}
	\end{array}
 \]
\caption{octagon again}\label{fig:OctagonDisk1}
\end{fg}
Let $\mcal{P}$ be the plane embedded in $\mcal{M}$, which is spanned by $p, q$ and some point $v$ in the interior of the segment $xy$.
Let $\gamma = \mcal{P} \cap \mcal{D}$, see Figure \ref{fig:OctagonCounterScheon}. Let $d:\gamma \mapsto \mathbb{R}^+$ be defined by sending each $t \in \gamma$ to its distance from $e$, where $e$ is the edge containing $p$ and $q$. Since $d(p) = d(q) = 0$, $d$ attains its maximum at some internal point of $\gamma$, denoted by $z$.
\begin{fg}[H]
\[ 
	\setlength{\unitlength}{0.5\standardunitlength}
	\begin{array}{c}  \hspace{-1.7mm}
		\raisebox{-8pt}{\input NormalOctagonB.tex }
		\hspace{-1.9mm}
	\end{array}
 \]
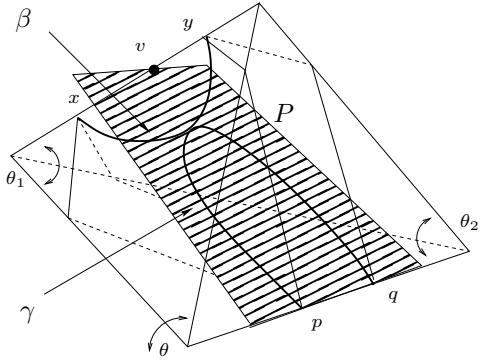
\caption{the plane $\mcal{P}$ intersects the octagon}\label{fig:OctagonCounterScheon}
\end{fg}
{\noindent In} order to proceed with the proof, the following definition by Banchoff of generalized curvature of a not-necessarily smooth curve, and a theorem by Rataj and Z$\ddot{a}$hle that relates the generalized curvature to the usual curvature of smooth curves, are needed:
\begin{df}(\cite{ban})\label{df:BanschofCurvature}
{\em Let $\alpha$ be a piecewise smooth curve embedded in a $3$-manifold $M$. Let $x \in \alpha$ be a point where $\alpha$ is not smooth, and let $t, t'$ be the two tangent vectors to $\alpha$ on the two sides of $x$ respectively. Let $\eta$ denote the angle between $t$ and $t'$. Then the curvature of $\alpha$ at $x$ is given by:}
\begin{eq}
k_{\alpha}(x) = \pi - \eta \;.
\end{eq}
\end{df}
{\noindent The} following theorem is given in \cite{raza} for all dimensions and for a large family of Riemannian manifolds. We will restrict it explicitly to curves and surfaces. See also \cite{cms2}, Theorem 6.
\begin{thm}\label{thm:ZhaleCurvatureConvergence}(\cite{raza}, Proposition $6$)
Let $L_1$, $L_2$ be two line segments in $\mathbb{R}^2$, so that $\L_1 \cap L_2$ is a point $x$. Let $N_{\epsilon} = \{ y \in \mathbb{R}^2; dist(y, L_1 \cup L_2) \leq \epsilon \}$ be an $\epsilon$-tubular neighborhood of $L_1 \cup L_2$. Let $x_{\epsilon}$ be a point on $\partial N_{\epsilon} = \{y \in \mathbb{R}^2; dist(y, L_1 \cup L_2) = \epsilon \}$, so that $dist(x, x_{\epsilon}) = \epsilon$, and let $k_{\epsilon}(x_{\epsilon})$ be the curvature of $\partial N_{\epsilon}$ at $x_{\epsilon}$, Then:
\begin{eq}
\lim_{\epsilon \rightarrow 0}k_{\epsilon}(x_{\epsilon}) = k_{(L_1 \cup L_2)}(x) \;.
\end{eq}
\end{thm}

{\noindent Having} the above definition and theorem, it is possible to proceed with the proof of Lemma \ref{lem:3Or4Gon}. Let $\beta$ be a curve on $\mcal{F}$ going from $x$ to $y$ that intersects $\gamma$ in a single point. Let $\theta$ be the dihedral angle of $\tau$ along the edge $e$. Recall that since the triangulation $\mcal{T}$ is $\varphi$-fat, and $\mcal{M}$ is non positively curved then $\theta \leq \pi - 2\varphi$. As a result, if $d(z) \rightarrow 0$ then there is a point on $\beta$ so that the curvature of $\beta$ at this point is at least $\pi - (\pi - 2\varphi) = 2\varphi$, which, is assumed to be greater than the Schoen curvature bound $C$. Thus there exists $r_0$ that depends only on the fatness $\varphi$ of $\mcal{T}$ and on the Schoen curvature bound $C$, so that $d(z) > r_0$.\\
{\noindent Let} $\mcal{O}$ be a circle in $\mcal{P}$ of curvature $C$ tangent to $\gamma$ at $z$, see Figure \ref{fg:ComparisonCircle}.

{\noindent Since} $\gamma$ is on the outside of $\mcal{O}$, at $z$, the points $p$ and $q$ at which $\gamma$ intersects the edge $e$, must also be outside $\mcal{O}$, for otherwise there would be some point on $\gamma$ at which the curvature of $\gamma$ is greater than the curvature of $\mcal{O}$ in contradiction to the curvature bound $C$. Denote by $e_{pq}$ the segment on $e$ from $p$ to $q$, and by $e_{p'q'}$ the chord in $\mcal{O}$ formed by the intersection points of $\mcal{O}$ with $e$. If $\mcal{M}$ is a Euclidean manifold then by the Pythagorean theorem the following holds:
\begin{eq}
\len(e_{pq}) \geq \len(e_{p'q'}) =  \sqrt{(\frac{1}{C}^2) - (\frac{1}{C} - d(z))^2} > \sqrt{(\frac{1}{C}^2) - (\frac{1}{C} - r_1)^2}\;\;.
\end{eq}
{\noindent In} the case where $\mcal{M}$ is hyperbolic manifold $\mcal{O}$ is a circle in the hyperbolic plane $\mcal{P}$, of hyperbolic radius $\sqrt{(\frac{1}{C}^2)}$. By properties of hyperbolic plane geometry it is also a Euclidean circle of Euclidean radius $\tanh \sqrt{(\frac{1}{C}^2)}$. Hence, a corresponding lower bound on the hyperbolic distance between $p$ and $q$ is deduced in this case as well.
\begin{fg}[H]
	\[ 
	\setlength{\unitlength}{0.1\standardunitlength}
	\begin{array}{c}  \hspace{-1.7mm}
		\raisebox{-8pt}{\input NonNormalFoldCurveComparisonCircle.tex }
		\hspace{-1.9mm}
	\end{array}
 \]
	\caption{a comparison circle}\label{fg:ComparisonCircle}
\end{fg}
{\noindent Similar} analysis will also hold for the points $x$ and $y$, while replacing the roles of $\gamma$ and $\beta$, therefore whenever $\lambda(\mcal{T}) < \len(e_{pq})$ there cannot be a normal disk with $\geq 8$ edges on its boundary. This proves Lemma \ref{lem:3Or4Gon}.
\end{proof}
\subsection{Proof of Theorem \ref{thm:LeastAreaIsNormal}}

\begin{proof}
In order to prove the theorem it is essential to show that under the assumptions, the given least-area surface $\mcal{F}$ satisfies the following:
\begin{enum}[(1)]
\item \label{NormalFact1} $\mcal{F} \cap int(\mcal{T}^{(2)})$ does not contain simple closed curves.
\item \label{NormalFact2} $\mcal{F} \cap \mcal{T}^{(3)}$ is a collection of disks.
\item \label{NormalFact3} All normal disks in $\mcal{F} \cap \mcal{T}^{(3)}$ are elementary disks.
\end{enum}
The proof of \ref{NormalFact1} was given in Lemma \ref{lem:NoSimpleClosedCurves} and the proofs of (\ref{NormalFact2}) and of (\ref{NormalFact3}) were done in Lemmas \ref{lem:IntesectionsAreDisks} and  \ref{lem:3Or4Gon}. This completes the proof of Theorem \ref{thm:LeastAreaIsNormal}.
\end{proof}

\section{Tameness of Non Normal Disks}
This section extends the study of the intersection of a least area surface with a fat triangulation. In particular, it focuses on the intersections of non normal disks, if any exist, with the tetrahedra of the triangulation.

The following theorem shows that when restricted to least area surfaces and to fat triangulations, the local picture is rather simple and controlled.
\renewcommand{\thethm}{\arabic{thm}}
\setcounter{thm}{1}
\begin{thm}\label{thm:NonNormalAreTame}
Let $\mcal{F}$ be an incompressible, $2$-sided least area surface in a closed, irreducible, orientable hyperbolic or Euclidean $3$-manifold $\mcal{M}$. Let $\Psi$ be the constant obtained in Theorem \ref{thm:LeastAreaIsNormal}, and let $\mcal{T}$ be a geodesic, least area, $\varphi$-fat triangulation of $\mcal{M}$ so that $\lambda(\mcal{T}) < \Psi$ and $\mcal{F}$ quasi normal with respect to $\mcal{T}$. Then every non normal disk in $\mcal{F} \cap \mcal{T}^{(3)}$ can be presented as a graph of a function defined on a disk that is contained in a single $2$-face of $\mcal{T}$.
\end{thm}
\renewcommand{\thethm}{\thesection.\arabic{thm}}
\setcounter{thm}{1}  

\begin{proof}
Let $\mcal{D}$ be some non normal disk properly embedded in a tetrahedron $\tau$, let $\mathfrak{f}$ be some $2$-face of $\tau$, and let $\delta \subset \mcal{D} \cap \mathfrak{f}$ be a bent curve in $\mcal{D} \cap \mathfrak{f}$ having its end points on the same edge $e$ of $\mathfrak{f}$. Let $p$ and $q$ be the two points consisting $\delta \cap e$. See Figure \ref{fg:MultipleBents} below for an illustration where $\delta$ has two components.
\begin{fg}[H]
\[ 
	\setlength{\unitlength}{0.5\standardunitlength}
	\begin{array}{c}  \hspace{-1.7mm}
		\raisebox{-8pt}{\input MultipleBentsCurve.tex }
		\hspace{-1.9mm}
	\end{array}
 \]
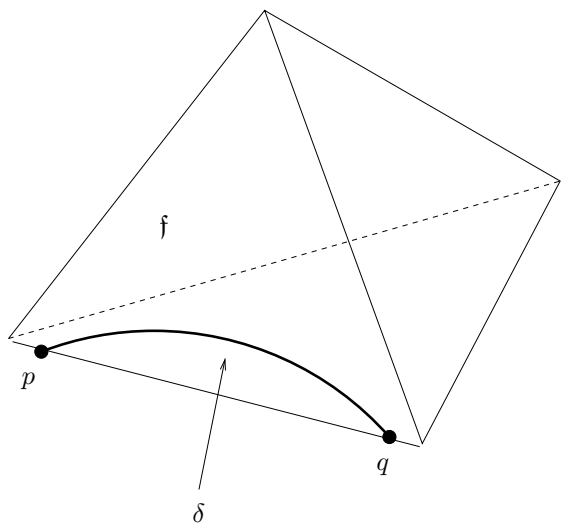
\caption{ A bent curve in the $2$-face $\mathfrak{f}$, and its two end points}\label{fg:MultipleBents}
\end{fg}
{\noindent Let} $\mathfrak{f}^+$ be the second $2$-face of $\tau$ that meets $\mathfrak{f}$ along $e$. There are three possibilities for the intersection curves of $\mcal{D} \cap \mathfrak{f}^+$ emanating from $p$ and $q$. Denote these curves by $c_1$ and $c_2$. Then these curves either connect each other without going through any other edge of $\mathfrak{f}^+$, where in this case $\mcal{D}$ is a graph over $\mathfrak{f}^+$, as illustrated in the following Figure \ref{fg:NonNormalSmallDisk}. In the other two possibilities these two curves may end at a single edge $e_1$ of $\mathfrak{f}^+$ other than $e$, or on each of the two edges $e_1$ and $e_2$ of $\mathfrak{f}^+$, that are not $e$. These two possibilities are illustrated in Figure \ref{fg:NonNormalEmanatingEdges}$(a)$ and $(b)$.
\begin{fg}[H]
\[ 
	\setlength{\unitlength}{0.5\standardunitlength}
	\begin{array}{c}  \hspace{-1.7mm}
		\raisebox{-8pt}{\input NonNormalSmallDisk.tex }
		\hspace{-1.9mm}
	\end{array}
 \]
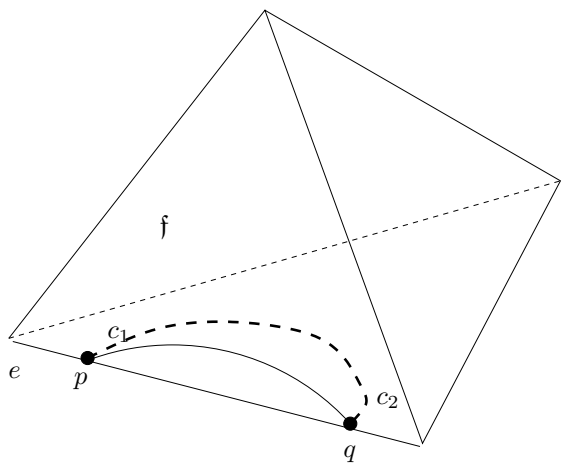
\caption{curves $c_1$ and $c_2$ connect each other inside $\mathfrak{f}^+$}\label{fg:NonNormalSmallDisk}
\end{fg}

\begin{fg}[H]
\[ 
	\setlength{\unitlength}{0.5\standardunitlength}
	\begin{array}{c}  \hspace{-1.7mm}
		\raisebox{-8pt}{\input NonNormalEmanatingEdges.tex }
		\hspace{-1.9mm}
	\end{array}
 \]
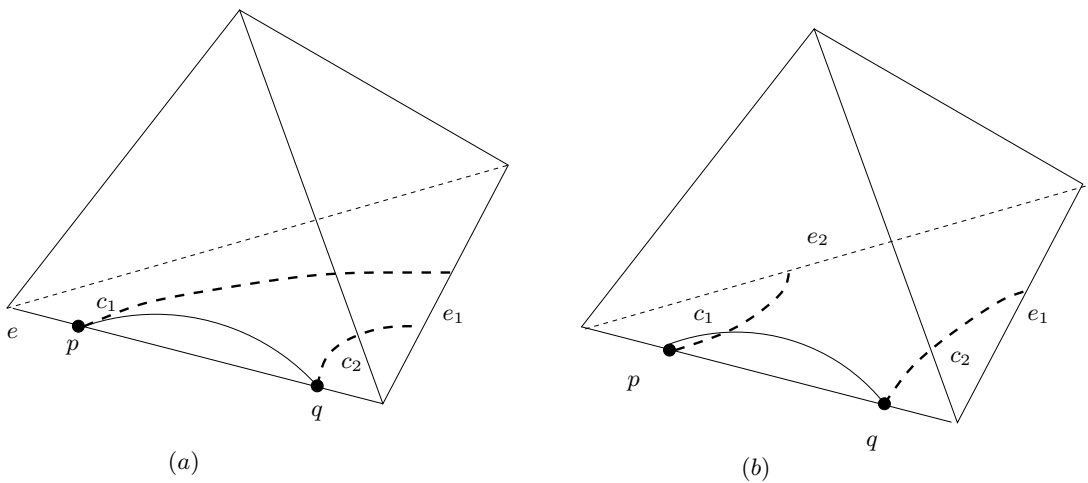
\caption{$c_1$ and $c_2$ ends on edges of $\mathfrak{f}^+$}\label{fg:NonNormalEmanatingEdges}
\end{fg}

{\noindent Suppose} the situation is as shown in Figure \ref{fg:NonNormalEmanatingEdges}$(a)$. Let $\mathfrak{f}'$ be the $2$-face of $\tau$ that meets $\mathfrak{f}^+$ along the edge $e_1$ , and consider the curves of $\mcal{D} \cap \mathfrak{f}'$ that emanate from the end points of $c_1$ and $c_2$. There are two possibilities. One is that these curves close each other without any further intersection with the edges of $\mathfrak{f}'$, as illustrated in Figure \ref{fg:NonNormalEmanatingSameII}$(a)$, and in this case the disk $\mcal{D}$ can be presented as a graph over a disk contained in the $2$-face $\mathfrak{f}^+$.

\begin{fg}[H]
\[ 
	\setlength{\unitlength}{0.5\standardunitlength}
	\begin{array}{c}  \hspace{-1.7mm}
		\raisebox{-8pt}{\input NonNormalEmanatingSame.tex }
		\hspace{-1.9mm}
	\end{array}
 \]
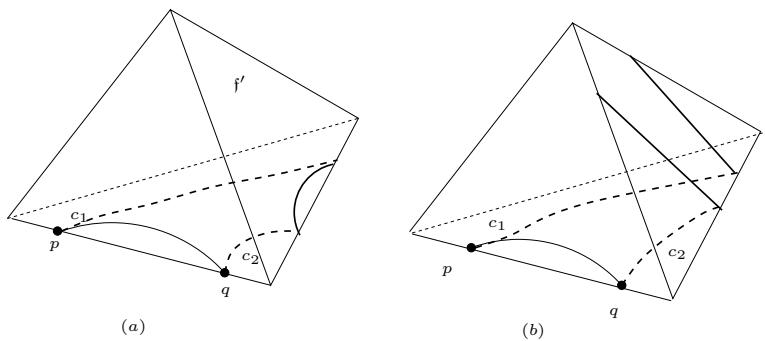
\caption{$c_1$ and $c_2$ ends at a single edge}\label{fg:NonNormalEmanatingSameII}
\end{fg}

{\noindent The} second option is that each of the curves ends at some point on $\partial \mathfrak{f}'$. In this latter case there are a few options, one of which is illustrated in Figure \ref{fg:NonNormalEmanatingSameII}(b). All other options are treated similarly. Let $p'$ and $q'$ be the points at which $c_1$ and $c_2$ intersect the edge $e_1$, and let $z$ be a point on the edge $e_3$ opposite to $e_1$, that is, the edge of $\tau$ that does not share a vertex with $e_1$. Let $\mcal{P}$ be the plane spanned by $p', q'$ and $z$. Let $\alpha$ be the intersection curve of $\mcal{P}$ and $\mcal{D}$, and let $\beta$ to be its co-curve on $\mcal{D}$, see Figure \ref{fg:NonNormalCompressingDisk}. Note that in this typical case $\mcal{D}$ cannot be given as a graph of a function defined on some sub disk contained in a single $2$-face of $\tau$.

\begin{fg}[H]
\[ 
	\setlength{\unitlength}{0.5\standardunitlength}
	\begin{array}{c}  \hspace{-1.7mm}
		\raisebox{-8pt}{\input NonNormalCompressingDisk.tex }
		\hspace{-1.9mm}
	\end{array}
 \]
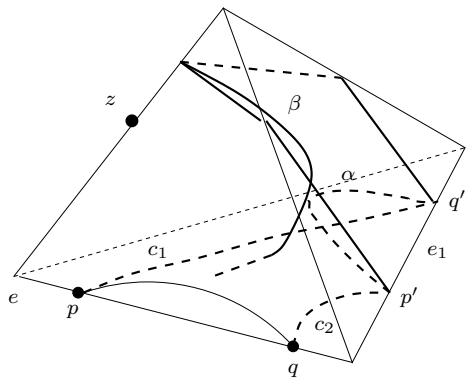
\caption{the curvature of $\mcal{D}$ along $\beta$ is too large}\label{fg:NonNormalCompressingDisk}
\end{fg}

{\noindent As} in the proof of Lemma \ref{lem:3Or4Gon}, if the distance between $\beta$ and the edge $e_1$ gets arbitrarily small, there will be a point on $\beta$ with curvature that is arbitrarily close to the generalised curvature defined in Definition \ref{df:BanschofCurvature}. Recall that this curvature is ($\pi - \textrm{dihedral angle of $\tau$ along $e_1$}$), and due to fatness of $\mcal{T}$ this curvature is greater than $2\varphi$. Since it is assumed that the Schoen curvature bound is smaller than $2\varphi$, there exists $r_0 > 0$ that depends only on $C$ and $\varphi$, so that the distance between $\beta$ and $e_1$ is greater than $r_0$. Hence, there exists $\Psi > 0$ that depends only on $C$ and $\varphi$ such that the distance between $p'$ and $q'$ is greater than $\Psi$. Therefore, if $\lambda(\mcal{T}) < \Psi$ there cannot be a non normal disk as in Figure \ref{fg:NonNormalEmanatingSameII}(b). Suppose now that we are in the situation shown in Figure \ref{fg:NonNormalEmanatingEdges}(b). Again, this case splits to two options. Figure \ref{fg:NonNormalEmanatingDifferentII}(a) shows one of these options, where the two curves $c_1$ and $c_2$ finally close each other at a point on the edge $e_4$, the opposite edge to $e$. In this case the disk $\mcal{D}$ can be projected into $\mathfrak{f}^+$. A typical example of the second option is depicted in Figure \ref{fg:NonNormalEmanatingDifferentII}(b) where the two curves following $c_1$ and $c_2$ intersect the two edges of $\mathfrak{f}$ that are different from $e$ before they close. In this case, let $p'$ and $q'$ be points on the two edges of $\mathfrak{f}$ that are not $e$, and let $V$ be the vertex of $\mathfrak{f}^+$ that is opposite to $e$, as shown in Figure \ref{fg:NonNormalCompDiskDifferent} .
\begin{fg}[H]
\[ 
	\setlength{\unitlength}{0.5\standardunitlength}
	\begin{array}{c}  \hspace{-1.7mm}
		\raisebox{-8pt}{\input NonNormalEmanatingPQonDiffferent.tex }
		\hspace{-1.9mm}
	\end{array}
 \]
\caption{$c_1, c_2$ go to different edges}\label{fg:NonNormalEmanatingDifferentII}
\end{fg}
\begin{fg}[H]
\[ 
	\setlength{\unitlength}{0.5\standardunitlength}
	\begin{array}{c}  \hspace{-1.7mm}
		\raisebox{-8pt}{\input NonNormalCompressingDiskDifferent.tex }
		\hspace{-1.9mm}
	\end{array}
 \]
\caption{$c_1, c_2$ go to different edges}\label{fg:NonNormalCompDiskDifferent}
\end{fg}
{\noindent Let} $P$ be the plane spanned by the points $p', q'$ and $V$, and let $\beta$ be the intersection curve $P \cap \mcal{D}$. Then there exists $r_0 > 0$ so that the distance between $\beta$ and the edge $e_4$ is greater than $r_0$, since otherwise the curvature of $\beta$ will exceed $2\varphi$ at some point on $\beta$, and $r_0$ depends only on $C$ and $\varphi$. As a result, there exists $\Psi > 0$ that depends only on $C$ and $\varphi$, so that the distance between $p'$ and $q'$ must be larger than $\Psi$. Again, this imposes that if the mesh of $\mcal{T}$ is smaller than $\psi$ there cannot be a non normal disk of the type shown in Figure \ref{fg:NonNormalEmanatingDifferentII}(b)\\

{\noindent This} completes the proof of Theorem \ref{thm:NonNormalAreTame}
\end{proof}
\subsection{Manifolds with Boundary}
This section generalizes Theorems \ref{thm:LeastAreaIsNormal} and \ref{thm:NonNormalAreTame} to compact $3$-manifolds with boundary. Doing so requires some modifications in the proofs of those theorems which will be given herein. Before stating the generalized theorem a care need be taken for the mere existence of least area surfaces in manifolds with boundary. For that, the route taken in ~\cite{hs} will be followed. 
\begin{df}
	A Riemannian three manifold $\mcal{M}$ has {\emph sufficiently convex} boundary if:
	\begin{enum}
		\item $\del \mcal{M}$ is piecewise smooth.
		\item Each smooth subsurface of $\del \mcal{M}$ has non-negative mean curvature with respect to the inward normal.
		\item There exists a Riemannian three manifold $\mcal{N}$ such that $\mcal{M}$ is isometric to a submanifold of $\mcal{N}$, and each smooth subsurface $\mcal{S}$ of $\del \mcal{M}$ extends to a smooth subsurface $\mcal{S}'$ embedded in $\mcal{N}$ such that $\mcal{S}' \cap \mcal{M} = \mcal{S}$.
	\end{enum}
\end{df} 
\noindent{Further}, the following definition given by Morrey in ~\cite{morr} is necessary in order to assure the existence of the curvature and injectivity radius bounds that are used in the proofs of Theorems  \ref{thm:LeastAreaIsNormal} and \ref{thm:NonNormalAreTame}.
\begin{df}
	A Riemnnian three manifold is {\emph homogeneously regular} if there exist positive constants $k, K$ so that every point of the manifolds lies in the image of a chart $\varphi$ with domain the unit ball $B(0, 1)$ in $\mathbb{R}^3$ such that \[ k||\nu||^2 \leq g_{ij}(\varphi (x))\nu_i \nu_j \leq K||\nu||^2\] for all $x \in B(0, 1)$,
	where $\nu$ is any tangent vector to $x$, and $g_{ij}$ are the components of the metric $g$ on $\mcal{M}$.
\end{df}
In ~\cite{my3} existence and uniqueness of least area disks spanned by simple closed curves in the boundary of a three manifold with sufficiently convex boundary are proven, see also ~\cite{hs}, Section $6$. In addition, Lemmas ~\ref{lem:DiskSpanCurve} and ~\ref{lem:ShortEnoughCurve} extend to homogeneously regular three manifolds with sufficiently convex boundary, see Section $6$ of ~\cite{hs}. Note that if $\mcal{T}$ is a triangulation of a three-manifold with boundary then $\mcal{T}$ induces a triangulation of $\partial \mcal{M}$, and according to ~\cite{hs} if $\mcal{M}$ is homogeneously regular with sufficiently convex boundary then the boundary of small enough $2$-face of $\mcal{T}$ that is embedded in $\partial \mcal{M}$ bounds a least area disk in $\mcal{M}$, that is either properly embedded in $\mcal{M}$ or a subset of $\partial \mcal{M}$. In the proof of Lemma \ref{lem:NoSimpleClosedCurves} we used the fact that $\mcal{T}$ is a least area triangulation, so every $2$-face is a least area disk. This assumption may not hold for three manifolds with boundary, for example in the case where a least area disk bounded by a simple closed curve included in $\partial \mcal{M}$ is properly embedded in $\mcal{M}$. 
As a result a proper modification of Theorem \ref{thm:LeastAreaIsNormal} needs to be stated. 
Before stating the appropriate theorem the following definition is required.
\begin{df}
	Let $\mcal{T}$ be a triangulation of a three manifold with boundary. A $2$-face of $\mcal{T}$ is called {\emph peripheral} if it has non empty intersection with $\partial \mcal{M}$. A $2$-face is called {\emph boundary $2$-face} if it is embedded in $\partial \mcal{M}$. A triangulation is called {\emph internally least area} if every $2$-face that is not a boundary $2$-face is a least area disk spanned by its boundary.
\end{df}
\noindent{In} order to simplify the statement of the following theorem the definition of quasi-normal surface is altered.
\begin{df}
	A surface $(\mcal{F}, \partial \mcal{F})$ properly embedded in a triangulated three manifold $(\mcal{M}, \partial \mcal{M}, \mcal{T})$ is {\emph internally quasi-normal} if its non-normal intersection with $\mcal{T}$ are bent curves, and a (possibly empty) collection of simple closed curves, each is contained in a boundary $2$-face of $\mcal{T}$, and bounds a least area disk properly embedded in a peripheral tetrahedron of $\mcal{T}$.
\end{df} 
\noindent{With} these definitions  at hand Theorem \ref{thm:LeastAreaIsNormal} is restated as follows:

\renewcommand{\thethm}{\arabic{thm}}
\setcounter{thm}{2}
\begin{thm} \label{thm:LeastAreaQuasiNormalBoundary}
	Let $\mcal{M}$ be an irreducible, orientable, homogeneously regular Euclidean or hyperbolic three manifold with sufficiently convex boundary. Let $\mcal{F}$ be a two-sided incompressible, $\partial$-incompressible, least area surface properly embedded in $\mcal{M}$. Let $\mcal{T}$ be some geodesic internally least area triangulation of $\mcal{M}$, so that $\mcal{F} \cap \mcal{T}^{(0)} = \emptyset$. Then for every $0 < \varphi < \frac{\pi}{2}$ there exists a constant $\Psi > 0$ that depends only on $inj_{\mcal{M}}, K_{\mcal{M}}$, and $\varphi$, such that if $\mcal{T}$ is $\varphi$-fat, and $\lambda(\mcal{T}) < \Psi$, then $\mcal{F}$ is internally quasi-normal with respect to $\mcal{T}$.
\end{thm}
\renewcommand{\thethm}{\thesection.\arabic{thm}}
\setcounter{thm}{1}
In order to prove Theorem \ref{thm:LeastAreaQuasiNormalBoundary} Lemma \ref{lem:NoSimpleClosedCurves} becomes the following lemma which in a sense resembles Lemma $6.7$ of ~\cite{hs}:
\begin{lem} \label{lem:NoSimpleClosedCurvesBoundary}
	There exists $\epsilon > 0$ such that if $\mcal{T}$ is $\varphi$-fat for some positive $\varphi$, and $\lambda(\mcal{T}) < \epsilon$ then the intersection of $\mcal{F}$ with the interior of any $2$-faces of $\mcal{T}$ that is not a boundary $2$-face does not contain simple closed curves, and if $\gamma \in \partial \mcal{F}$ is a simple closed curve contained in the interior of a boundary $2$-face of $\mcal{T}$ then $\gamma$ bounds a least area disk that is either contained in $\partial \mcal{M}$ or properly embedded in $mcal{M}$ 
\end{lem}
\begin{proof}
	Let $\frak{f}$ be a $2$-face of $\mcal{T}$. If $\frak{f}$ is not a boundary $2$-face then all arguments in the proof of Lemma ~\ref{lem:NoSimpleClosedCurves} hold, so the proof is identical. Suppose $\frak{f}$ is a boundary $2$-face, and that $\gamma \in \partial \mcal{F}$ is a simple closed curve contained in the interior of $\frak{f}$. Let $\mcal{N} = \mcal{N}(\delta, \gamma, \mcal{F})$ be a $\delta$-collar neighborhood of $\gamma$ in $\mcal{F}$, for some small enough $\delta$ so that $ \mcal{N}$ is entirely contained in a tetrahedron $\tau \in \mcal{T}$ of which $\frak{f}$ is a $2$-face. Let $\Gamma \in \partial \mcal{N}$ be the parallel of $\gamma$ in the boundary of $\mcal{N}$. Then $\Gamma$ is a simple closed curve in $\mcal{F}$ contained in the $3$-ball $\tau$, and since $\mcal{F}$ is incompressible $\Gamma$ bounds a disk $\mcal{D}$ in $\mcal{F}$. As a result $\gamma$ is the boundary of a disk $\mcal{E} = \mcal{D} \cup \mcal{N}$ that is contained in $\mcal{F}$. If the disk $\mcal{D}$ is not entirely contained in $\tau$ then the situation is identical to case II of the proof of Lemma ~\ref{lem:NoSimpleClosedCurves}, and the arguments used in this case still hold, so the proof proceeds identically. If $\mcal{D}$ is contained in $\tau$ then the interior of $\mcal{E}$ is contained in $\tau$ as well. In addition, $\gamma$ bounds a disk $\mcal{D}'$ in the $2$-face $\frak{f}$. Since $\mcal{F}$ is embedded in $\mcal{M}$, the disk $\mcal{D}$ is embedded, and as a result $\mcal{E}$ is embedded as well. Hence, $\mcal{D}' \cup \mcal{E}$ is a $2$-sphere embedded in $\mcal{M}$. Therefore, since $\mcal{M}$ is irreducible the disk $\mcal{E}$ there is an isotopy of $\mcal{F}$ mapping $\mcal{E}$ isotopically to $\mcal{D}'$, and which is the identity elsewhere. If this isotopy reduces the area of $\mcal{F}$ then the disk $\mcal{D}'$ is a least area disk bounded by $\gamma$, otherwise, according to lemma $6.6$ of ~\cite{hs} $\gamma$ bounds least area disk that is properly embedded in the tetrahedron $\tau$ which, by definition, is a peripheral tetrahedron, and the proof of the lemma is completed.
\end{proof}
\noindent{The} following two lemmas are the appropriate versions of Lemmas \ref{lem:IntesectionsAreDisks}, and \ref{lem:3Or4Gon} for the case $\mcal{M}$ has non-empty boundary.
\begin{lem}\label{lem:IntesectionsAreDisksBoundary}
	Let $\mcal{M}$ be homogeneously regular three manifold with sufficiently convex boundary. There exists $\epsilon > 0$ such that if $\lambda(\mcal{T}) < \epsilon$, then the intersection of $\mcal{F}$ with any tetrahedron of $\mcal{T}$ is a (possibly empty) collection of disks.
\end{lem}

\begin{lem}\label{lem:3Or4GonBoundary}
	Let $\mcal{M}$ be homogeneously regular three manifold with sufficiently convex boundary. For every $0 < \varphi < \frac{\pi}{2}$ there exists $\epsilon > 0$ that depends only on $inj_{\mcal{M}}, K_{\mcal{M}}$ and $\varphi$, such that if $\lambda(\mcal{T}) < \epsilon$, and $\mcal{T}$ is $\varphi$-fat, then every normal disk in $\mcal{F} \cap \mcal{T}$ is either a $3$-gon or a $4$-gon.
\end{lem}
\noindent{Since} the proofs of both Lemmas \ref{lem:IntesectionsAreDisks}, and \ref{lem:3Or4Gon} heavily relies on the curvature bound of Schoen, and since, according to \cite{morr} being homogeneously regular is equivalent to having global bound on the curvature in the case of three manifold with boundary, the proof of Lemma \ref{lem:IntesectionsAreDisksBoundary} remains the same as the proof of Lemma \ref{lem:IntesectionsAreDisks}, and the proof of Lemma \ref{lem:3Or4GonBoundary} is identical to that of \ref{lem:3Or4Gon}. 
One thing to notice however, is the following:

\begin{rem}\label{rem:NoPeripheralBentCurves}
	If $\frak{f}$ is a peripheral $2$-face that is not a boundary $2$-face, so that $\frak{f}$ meets $\partial \mcal{M}$ transversally along an edge $e$, then there cannot be any bent curve in $\frak{f}$ having both end points on $e$. 
	This is because the surface $\mcal{F}$ is $\partial$-incompressible in $\mcal{M}$, and any such bent curve gives rise to $\partial$-compression. 
\end{rem}

\noindent{The} proof of Theorem \ref{thm:LeastAreaQuasiNormalBoundary} follows from Lemmas \ref{lem:NoSimpleClosedCurvesBoundary}, \ref{lem:IntesectionsAreDisksBoundary}, and \ref{lem:3Or4GonBoundary} exactly the same way as Theorem \ref{thm:LeastAreaIsNormal} follows from Lemmas \ref{lem:NoSimpleClosedCurves}, \ref{lem:IntesectionsAreDisks}, and \ref{lem:3Or4Gon}

\renewcommand{\thethm}{\arabic{thm}}
\setcounter{thm}{3}
\begin{thm}\label{thm:NonNormalAreTameBoundary}
	Let $\mcal{F}$ be an incompressible, $\partial$-incompressible, $2$-sided least area surface in a homogeneously regular, irreducible, orientable hyperbolic or Euclidean $3$-manifold $\mcal{M}$, with sufficiently convex boundary. Let $\Psi$ be the constant obtained in Theorem \ref{thm:LeastAreaIsNormal}, and let $\mcal{T}$ be a geodesic, internally least area, $\varphi$-fat triangulation of $\mcal{M}$ so that $\lambda(\mcal{T}) < \Psi$ and $\mcal{F}$ quasi normal with respect to $\mcal{T}$. Then every non normal disk in $\mcal{F} \cap \mcal{T}^{(3)}$ can be presented as a graph of a function defined on a disk that is contained in a single $2$-face of $\mcal{T}$.
\end{thm}
\renewcommand{\thethm}{\thesection.\arabic{thm}}
\setcounter{thm}{1}
\begin{proof} 
	The proof of the theorem is similar to the proof of Theorem \ref{thm:NonNormalAreTame} which completely relies on the existence of global bound on the curvature of stable minimal surfaces in a closed three manifold. Again, homogeneous regularity of the manifold imposes the existence of such bound. In light of Remark \ref{rem:NoPeripheralBentCurves} the proof is simplified as some possible cases are actually impossible and need not be considered.
\end{proof}

\section{Fat Subdivision of a Triangulation}
In this section a method is developed for iteratively subdividing a given triangulation of a closed, non-positively curved $3$-manifold, in order to obtain a sequence of triangulations with decreasing mesh and fixed fatness. Specifically, the following lemma is the main result of this section. The lemma will be used later in Section $5$ for the proof of Theorem \ref{thm:AppxThm}.
\begin{lem}\label{lem:FatSubdiv}
Let $(\mcal{M}, \mcal{T})$ be a triangulated, closed, non-positively curves $3$-manifold. Then there exists a sequence $\{\mcal{T}_j\}_{j \in \mathbb{N}}$ so that $\lambda_j = \lambda(\mcal{T}_j) \rightarrow 0$, as $j \rightarrow \infty$, and such that $\mcal{T}_j$ is $\varphi$-fat for some fixed $\varphi$ for all $j$.
\end{lem}
{\noindent As} natural as this idea may look, some care must be taken, since not every subdividing method will maintain a fixed fatness. For instance, if we use barycentric subdivision, the dihedral angles decrease in a magnitude of half at each iteration so fatness cannot be kept fixed. As a warm up, we will start with a $2$-dimensional Euclidean triangulation and then extend the definition to regular Euclidean $3$-simplices. Afterwards we will further extend the method to triangulated $3$-manifolds in general.

\begin{df}\label{df:Median2Dim}
{\em Let $\sigma \subset \mathbb{R}^2$ be a triangle. Define its} median subdivision $Med_1(\sigma)$, {\em to be the cell complex obtained by adding a vertex at the middle point of every edge of $\sigma$ and connecting two newly added vertices by an edge. If $\mcal{T}$ is a $2$ dimensional triangulation, then $Med(\mcal{T})$ is obtained by taking $Med(\sigma)$ for every $\sigma \in \mcal{T}$. The $n$-th subdivision of $\mcal{T}$ is defined iteratively by $Med_{(n)}(\mcal{T}) = Med(Med_{(n-1)}(\mcal{T}))$. See Figure \ref{fig:Median2Dim}, for $n = 2$}.
\begin{fg}[H]
\[ 
	\setlength{\unitlength}{0.5\standardunitlength}
	\begin{array}{c}  \hspace{-1.7mm}
		\raisebox{-8pt}{\setlength{\unitlength}{0.00025000in}
\begingroup\makeatletter\ifx\SetFigFont\undefined%
\gdef\SetFigFont#1#2#3#4#5{%
  \reset@font\fontsize{#1}{#2pt}%
  \fontfamily{#3}\fontseries{#4}\fontshape{#5}%
  \selectfont}%
\fi\endgroup%
{\renewcommand{\dashlinestretch}{30}
\begin{picture}(5949,4227)(0,-10)
\path(2487,1500)(3237,2475)(3912,1575)(2487,1500)
\path(1887,2400)(4662,2550)(3087,525)(1812,2400)
\path(912,1425)(2487,1500)(1512,525)(912,1425)
\path(12,525)(3612,4200)(5937,600)(12,525)
\path(3912,1575)(4587,600)(5187,1725)(3912,1575)
\path(3237,2475)(2637,3225)(4137,3450)(3237,2475)
\put(1737,0){\makebox(0,0)[lb]{{\SetFigFont{5}{6.0}{\rmdefault}{\mddefault}{\updefault}Median subdivision}}}
\end{picture}
} }
		\hspace{-1.9mm}
	\end{array}
 \]
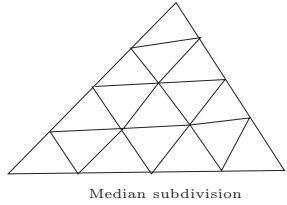
\caption{second (n=2) median subdivision in dimension $2$}\label{fig:Median2Dim}
\end{fg}
{\noindent In} the sequel we will omit the index $1$, and use the notation $Med(\sigma)$.
\end{df}
\begin{rem}
{\em Note that if $\mcal{T}$ is a} regular triangulation {\em (i.e., all its edges are equilateral), then by  similarity of triangles $Med(\mcal{T})$ is also regular. Hence, $Med_{(n)}(\mcal{T})\; ; n \geq 1$, are all $\varphi$-fat with the same fatness coefficient $\varphi = \frac{\pi}{3}$. In general, for every simplex of $Med_{(n)}(\mcal{T})$, its angles are the same as those of the original simplex it is contained in. Hence, again fatness of the original triangulation is kept fixed.
}
\end{rem}
\begin{df}\label{df:MedComp}
{\em Let $\tau$ be a regular tetrahedron in $\mathbb{R}^3$. The} median complex of $\tau$, {\em denoted by $MedComp(\tau)$, is defined as follows: First, subdivide each of its $2$-faces according to Definition \ref{df:Median2Dim}. As a result, each vertex of $\tau$ is separated from all other vertices by a triangle in $\del \tau$ that is formed by three edges of $Med(\del \tau)$. The median complex of $\tau$ is the complex obtained by adding the planar triangular $2$-face spanned by this triangle to the set of $2$-faces. The set of $3$-cells of $MedComp(\tau)$ is made of four tetrahedra and one octahedron. See Figure \ref{fig:MedianComplex}.
}
\begin{fg}[H]
\[ 
	\setlength{\unitlength}{0.5\standardunitlength}
	\begin{array}{c}  \hspace{-1.7mm}
		\raisebox{-8pt}{\setlength{\unitlength}{0.00041667in}
\begingroup\makeatletter\ifx\SetFigFont\undefined%
\gdef\SetFigFont#1#2#3#4#5{%
  \reset@font\fontsize{#1}{#2pt}%
  \fontfamily{#3}\fontseries{#4}\fontshape{#5}%
  \selectfont}%
\fi\endgroup%
{\renewcommand{\dashlinestretch}{30}
\begin{picture}(5199,3939)(0,-10)
\thicklines
\dashline{120.000}(1362,2412)(2637,1737)(1887,462)
\dashline{120.000}(2637,1737)(4587,1437)
\dashline{120.000}(2637,1737)(4137,3087)
\path(1887,462)(4587,1437)
\path(1362,2412)(4212,3087)
\thinlines
\dashline{60.000}(12,912)(5187,2562)
\thicklines
\path(1362,2412)(3312,1887)(1887,462)(1362,2412)
\path(3312,1887)(4212,3087)(4587,1437)(3312,1887)
\thinlines
\path(2712,3912)(12,912)(3912,12)
	(2712,3912)(5187,2562)(3912,12)
\end{picture}
} }
		\hspace{-1.9mm}
	\end{array}
 \]
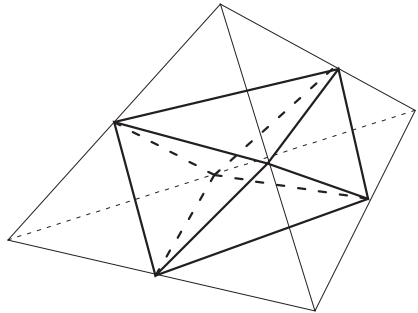
\caption{the complex $MedComp(\tau)$}\label{fig:MedianComplex}
\end{fg}
\end{df}
{\noindent In} order to define the {\em median simplicial subdivision} of $\tau$, the octahedron will be dissected into tetrahedra.
Consider the three squares $(a, b, f, c);\;(d, b, e, c)$ and $(a, d, f, e)$ as indicated in Figure \ref{fg:RegularToAlmostRegular}. Adding any two of these squares to $MedComp(\tau)$ decomposes the octagon into four tetrahedra. The obtained simplicial complex is defined as the {\em median subdivision} of $\tau$, and it is denoted by $Med(\tau)$.
\begin{df}\label{df:AlmostRegular}
{\em
A tetrahedron is called} almost regular {\em if it has five equilateral edges, while the sixth edge is $\sqrt{2}$ times longer. A triangulation $\mcal{T}$ is called almost regular if every tetrahedron of $\mcal{T}$ is either regular or almost regular.
}
\end{df}
\begin{clm}\label{clm:RegToAlmostReg}
If $\tau$ is a regular tetrahedron then $Med(\tau)$ is an almost regular triangulation.
\end{clm}
\begin{proof}
First note that the fact that $\sigma$ is regular, assures that the vertices of any of the squares $(a, b, f, c);\;(d, b, e, c)$ and $(a, d, f, e)$ are coplanar. Second, notice that $Med(\tau)$ is composed of two types of tetrahedra: Tetrahedra that are adjacent to the vertices of $\tau$ (called {\em peripheral tetrahedra}); and tetrahedra that are formed from the dissection of the octahedron. These four tetrahedra will be called {\em internal tetrahedra}. Suppose the edge length of $\tau$ is $\len$. From the definition of $Med(\tau)$ the length of any edge of a peripheral tetrahedron is $\frac{1}{2}\len$, therefore all peripheral tetrahedra are regular.
\begin{fg}[H]
\[ 
	\setlength{\unitlength}{0.5\standardunitlength}
	\begin{array}{c}  \hspace{-1.7mm}
		\raisebox{-8pt}{\setlength{\unitlength}{0.00041667in}
\begingroup\makeatletter\ifx\SetFigFont\undefined%
\gdef\SetFigFont#1#2#3#4#5{%
  \reset@font\fontsize{#1}{#2pt}%
  \fontfamily{#3}\fontseries{#4}\fontshape{#5}%
  \selectfont}%
\fi\endgroup%
{\renewcommand{\dashlinestretch}{30}
\begin{picture}(5199,3939)(0,-10)
\put(2187,1587){\makebox(0,0)[lb]{\smash{{\SetFigFont{9}{10.8}{\rmdefault}{\mddefault}{\updefault}$c$}}}}
\thicklines
\dashline{120.000}(2637,1737)(4587,1437)
\dashline{120.000}(2637,1737)(4137,3087)
\put(987,2562){\makebox(0,0)[lb]{\smash{{\SetFigFont{9}{10.8}{\rmdefault}{\mddefault}{\updefault}$a$}}}}
\put(3537,1887){\makebox(0,0)[lb]{\smash{{\SetFigFont{9}{10.8}{\rmdefault}{\mddefault}{\updefault}$b$}}}}
\put(1737,162){\makebox(0,0)[lb]{\smash{{\SetFigFont{9}{10.8}{\rmdefault}{\mddefault}{\updefault}$d$}}}}
\dashline{120.000}(1362,2412)(2637,1737)(1887,462)
\thinlines
\dashline{60.000}(12,912)(5187,2562)
\thicklines
\path(1362,2412)(3312,1887)(1887,462)(1362,2412)
\path(3312,1887)(4212,3087)(4587,1437)(3312,1887)
\path(1887,462)(4587,1437)
\path(1362,2412)(4212,3087)
\put(4812,1362){\makebox(0,0)[lb]{\smash{{\SetFigFont{8}{9.6}{\rmdefault}{\mddefault}{\updefault}$f$}}}}
\put(4362,3312){\makebox(0,0)[lb]{\smash{{\SetFigFont{8}{9.6}{\rmdefault}{\mddefault}{\updefault}$e$}}}}
\thinlines
\path(2712,3912)(12,912)(3912,12)
	(2712,3912)(5187,2562)(3912,12)
\end{picture}
} }
		\hspace{-1.9mm}
	\end{array}
 \]
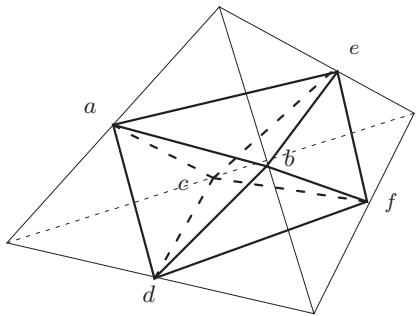
\caption{median subdivision of a regular $3$-simplex}\label{fg:RegularToAlmostRegular}
\end{fg}
{\noindent Now} let $\sigma'$ denote some internal tetrahedron. Notice that $\sigma'$ shares five of its edges with peripheral tetrahedra. Therefore it has five equilateral edges of length $\frac{\len}{2}$. The sixth edge of $\sigma'$ is formed from the dissection of the octahedron, and it is a diagonal of one of the squares mentioned above, of edge length $\frac{\len}{2}$, so its length equals $\frac{\sqrt{2}\len}{2}$ and the proof is completed.\\
\end{proof}
{\noindent In} order to proceed by iteratively taking median subdivisions, it is essential to describe explicitly the median subdivision of an almost regular tetrahedron:
Let $\tau$ be an almost regular tetrahedron as shown in Figure \ref{fig:MedianSubdivAlmostRegular}. Then $\tau$ has five equilateral edges and one longer edge that is $\sqrt{2}$-times longer than all other edges. Let $L$ denote this longer edge(dashed edge in the figure).
\begin{fg}[H]
\[ 
	\setlength{\unitlength}{0.5\standardunitlength}
	\begin{array}{c}  \hspace{-1.7mm}
		\raisebox{-8pt}{\input AlmostRegularTwoQuads.tex }
		\hspace{-1.9mm}
	\end{array}
 \]
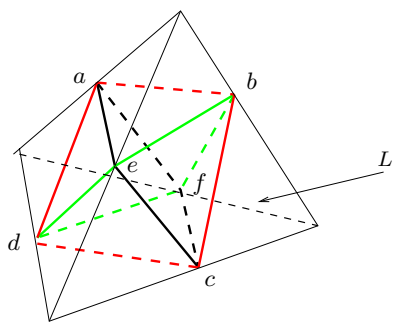
\caption{median subdivision of an almost regular $3$-simplex}\label{fig:MedianSubdivAlmostRegular}
\end{fg}
{\noindent Since} all edges except $L$ are equilateral, we get that the $4$-gon, $(a, b, c, d)$ indicated in the figure, is a rectangle, and as in the case of regular tetrahedron, we also have that the vertices of the $4$-gons $(a, e, c, f)$ and $(d, e, b, f)$ are coplanar. Note that for a general tetrahedron this may not hold.

{\noindent Let} $\rho = \len(ad)$ denote the length of the edge $(ad)$. We have the following equalities:
\[ \len(ac) = \len(af) = \len(fc) = \len(ce) = \rho \;\; ,\]
\[ \len(eb) = \len(bf) = \len(fd) = \len(de) = \rho \;\; ,\]
\[ \len(ad) = \len(bc) = \rho \;\; ,\]
\[ \len(dc) = \len(ab) = \sqrt{2}\rho \;\; .\]
{\noindent The} interior octahedron is dissected into four tetrahedra by explicitly using the two quadrilaterals $(a, e, c, f)$ and $(d, e, b, f)$.
\begin{fg}[H]
\[ 
	\setlength{\unitlength}{0.5\standardunitlength}
	\begin{array}{c}  \hspace{-1.7mm}
		\raisebox{-8pt}{\input AlmostRegularDissectionII.tex }
		\hspace{-1.9mm}
	\end{array}
 \]
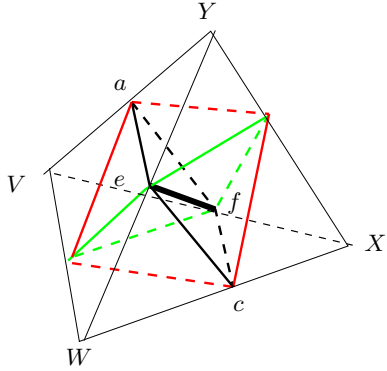
\caption{\small{dissection of an interior octahedron of the median complex of an almost regular tetrahedron}}\label{fig:MedianSubdivAlmostRegDissection}
\end{fg}
{\noindent The} simplicial subdivision described above is $Med(\tau)$.
\begin{clm}\label{clm:AlmostRegToAlmostReg}
If $\tau$ is an almost regular tetrahedron, then $Med(\tau)$ is an almost regular triangulation.
\end{clm}
\begin{proof}
As in the proof of Claim \ref{clm:RegToAlmostReg}, the tetrahedra of $Med(\tau)$ are divided into subsets of peripheral and internal tetrahedra. Since the median subdivision is defined via splitting each edge in the middle, any peripheral tetrahedron has five equilateral edges and one that is $\sqrt{2}$ longer, hence any such tetrahedron is almost regular. Let $\sigma'$ again denote an internal tetrahedron of $Med(\tau)$. As before, $\sigma'$ shares five of its edges with peripheral tetrahedra. Therefore, it either has five equilateral edges of length $\rho$, or four equilateral and one that is $\sqrt{2}$ longer, so the lengths of its edges are either $\rho$ or $\sqrt{2}\rho$. The sixth edge of $\sigma'$, the segment $ef$ in the figure above, is again formed as the intersection of the two quadrilaterals we added. Each of these quads is a rhombus with edge length $\rho$, and the edge $ef$ is the shorter diagonal of the rhombus. In order to compute the length of $ef$ we observe that the longer diagonal is also a diagonal of the rectangle $(a, b, c, d)$ whose edge lengths are $\rho$ and $\sqrt{2}\rho$, so its length is $\sqrt{3}\rho$. In addition, let $X, Y, V ,W$ be vertices of the original almost regular tetrahedron as depicted in Figure \ref{fig:MedianSubdivAlmostRegDissection}, then due to similarity of triangles one has that $\len(ae) = \len(fc) = \frac{1}{2}\len(VW)$ and $\len(ec) = \len(af) = \frac{1}{2}\len(YX)$, and since $\len(VW) = \len(YX)$, the quadrilateral $(a, e, c, f)$ is a rhombus of edge length $\rho$, such that $ef$ is its short diagonal, and its long diagonal has length $\sqrt{3}\rho$. From the Pythagorean theorem one immediately gets that $\len(ef) = \rho$. As a result, we obtain that $\sigma'$ either has six edges of length $\rho$ or five edges of length $\rho$ and one edge of length $\sqrt{2}\rho$. This completes the proof of the claim.
\end{proof}

\begin{df}\label{df:NthMedianSubdiv}
{\em For $n > 1$ we define the $n^{th}$ median subdivision of a regular tetrahedron, iteratively by:
\begin{eq}\label{eq:NthMedianSubdivRegular}
Med_{(n)}(\tau) = Med(Med_{(n-1)}(\tau)) \;.
\end{eq}
}
\end{df}

\begin{lem}\label{lem:MadianSubdivKeepsFatness}
Let $\sigma$ be a regular tetrahedron. Then, $\{Med^{(n)}(\tau)\; ;n \geq 1\}$ are $\varphi$-fat triangulations with a fixed fatness coefficient $\varphi$ for all $n$.
\end{lem}

\begin{proof}
It suffices to show that for every $n \geq 1$, $\{Med_{(n)}(\tau)\; ;n \geq 1\}$ is an almost regular triangulation. This is due to the fact that if this is true then the ratio between the lengths of any two edges of $\{Med_{(n)}(\tau)\}$ is either $1$ or $\sqrt{2}^{\pm1}$. The proof of the lemma is done by induction. The case $n = 1$ was shown in Claim \ref{clm:RegToAlmostReg}. Suppose $Med_{(k)}(\sigma)$ is almost regular for all $k \leq n-1$. Since $Med_{(n)}(\sigma)$ is the $1$-step subdivision of $Med_{(n-1)}(\sigma)$ the induction step follows directly from Claim \ref{clm:AlmostRegToAlmostReg}.
\end{proof}
{\noindent In} order to define the median subdivision of a non positively curved $3$-manifold the following lemma from \cite{bre} is needed.
\begin{lem}[\cite{bre}, Lemma $1$]\label{lem:BreslinBilipschitz}
Let $\mcal{M}$ be a non positively curved $3$-manifold. Let $\mcal{T}$ be a $\varphi$-fat triangulation of $\mcal{M}$. Then every $3$-simplex of $\mcal{T}$ is bilipschitz diffeomorphic to a regular Euclidean simplex in $\mathbb{R}^3$, where the bilipschitz constant depends only on $K_{\mcal{M}}, inj_{\mcal{M}}$ and $\varphi$.
\end{lem}
{\noindent It} is now possible to define the median subdivision of a non positively curved $3$-manifold.
\begin{df}\label{df:MedianSubDivManifold}
{\em Let $(\mcal{M}, \mcal{T})$ be a non positively curved triangulated $3$-manifold. Then according to Lemma \ref{lem:BreslinBilipschitz}, every simplex $\sigma$ of $\mcal{T}$ is a bilipschitz image of the form $\sigma = f(\mcal{E})$, where $\mcal{E}$ is a regular Euclidean $3$-simplex. The $n^{th}$ median subdivision of $\sigma$, $Med^{(n)}(\sigma)$ for $n \geq 1$, is defined as $f(Med^{(n)}(\mcal{E}))$. The median subdivision of $\mcal{T}$ is given by the subdivision of each of its simplices.
}
\end{df}
{\noindent The} following is a restatement of Lemma \ref{lem:FatSubdiv}, is the main result of this section.
\begin{lem} \label{lem:MedianSubdivPreservesFatness}
Let $\mcal{M}$ be a closed, non-positively curved $3$-manifold, and let $\mcal{T}$ be a triangulation of $\mcal{M}$. Then $\big \{\mcal{T}_n = Med_{(n)}(\mcal{T})\big\}_{n \in \mathbb{N}}$ is a sequence of $\varphi$-fat triangulations, for a fixed $0 < \varphi < 2\pi$, such that $\lambda(\mcal{T}_n) \rightarrow 0$ as $n \rightarrow \infty$.
\end{lem}
\begin{proof}
Since $\mcal{M}$ is compact, $\mcal{T}$ is a $\varphi_0$-fat for some $\varphi_0$.
Since, according to Lemma \ref{lem:MadianSubdivKeepsFatness}, $Med_{(n)}(|K|)$ is an almost regular triangulation,
we have that if $e, e'$ are any two edges of $\mcal{T}_n$ then, following \cite{bre} we have:
\begin{eq}\label{eq:EdgeRatioManifold}
|\frac{l(e)}{l(e')}| \leq (C\cdot\sqrt{2})^{\epsilon} ; \epsilon = \pm 1 \;,
\end{eq}
{\noindent where} $C$ is the bilipschitz constant. From Equation \ref{eq:EdgeRatioManifold} it follows that there exists fixed $\varphi = \varphi(C)$ so that $\mcal{T}_n$ is $\varphi$-fat for every $n \in \mathbb{N}$.
Let $\lambda_n = \lambda(\mcal{T}_n)$, then the following also holds:
\begin{eq}
\frac{1}{C}\lambda(Med_{(n)}(\mcal{T}) < \lambda_n < C\lambda(Med_{(n)}(\mcal{T})\;.
\end{eq}
{\noindent Hence}, $\lambda_n \rightarrow 0$ as $n \rightarrow \infty$, and the Lemma is proved.
\end{proof}

\section{Approximation of Least Area Surfaces in $\mcal{M}$}
In this section Theorem \ref{thm:AppxThm} will be proved for least area incompressible surfaces embedded in a non positively curved triangulated $3$-manifold $(\mcal{M}, \mcal{T}, \mathfrak{g})$.

\subsection{Notations and Definitions}
{\noindent Since} it is desired that the Approximation of a least area surface by some sequence of surfaces will also approximate the area of the surface by the areas of the surfaces in the sequence, and in view of Example \ref{exm:Lantern} it is evident that convergence with respect to some metric distance will not suffice. We will use the following of convergence in order to approximate the least area surface $\mcal{F}$, as well as its area.
\begin{df}\label{df:Convergence}
{\em
A sequence $\mcal{G}_j$ of closed embedded $\mcal{C}^r \; ; r \geq 2$, surfaces in $\mcal{M}$} converges {\em to a given closed embedded $\mcal{C}^r$ surface $\mcal{G}$, if:
\begin{enum}
\item \[ d_H(\mcal{G}_j, \mcal{G}) \rightarrow 0 \;, \textrm{as $j \rightarrow \infty$} \;,\]
{\noindent where} $d_H$ is the Hausdorff metric as defined in \ref{df:HausdorffMetric}.
\item Let $\{x_i \in \mcal{G}_i\}$ be a sequence of points so that $x_i \rightarrow x$ for some $x \in \mcal{G}$, and let $N_{x_i}, N_{x}$ denote the corresponding unit normal vectors to $\mcal{G}_i$ at $x_i$ and to $\mcal{G}$ at $x$ respectively. Then:
\[ \measuredangle(N_{x_j}\mcal{G}_j, N_x\mcal{G}) \rightarrow  0 \; , \]
{\noindent where} $\measuredangle(u, v)$ denotes the angle between any two unit vectors $u$ and $v$.
\end{enum}
}
\end{df}

\begin{df}
{\em
Let $\mcal{D}$ be a non normal disk in $\mcal{F} \cap \mcal{T}^{(3)}$, embedded in some tetrahedron $\tau$, and let $\mathfrak{f}$ be the $2$-face $\mcal{D}$ is projected into as in Lemma \ref{thm:NonNormalAreTame}. Then $\mcal{D}$} has a corner {\em if there exists a vertex $v$ of $\mathfrak{f}$ that is separated from all other vertices of $\tau$ by a curve that is made of two normal arcs as shown in Figure \ref{fg:Corner} below.
\begin{fg}[H]
\[ 
	\setlength{\unitlength}{0.5\standardunitlength}
	\begin{array}{c}  \hspace{-1.7mm}
		\raisebox{-8pt}{\input NonNormalCorner.tex }
		\hspace{-1.9mm}
	\end{array}
 \]
\caption{a corner of a non normal disk}\label{fg:Corner}
\end{fg}
}
\end{df}

{\noindent The} following definition presents the prototype of surfaces by which $\mcal{F}$ will be approximated. Following Theorem \ref{thm:LeastAreaIsNormal} without loss of generality assume that $\mcal{T}$ is some given fat triangulation of some given fatness $\varphi$, so that $\lambda(\mcal{T})$ is small enough so that $\mcal{F}$ is quasi normal with respect to $\mcal{T}$.
\begin{df}\label{df:FlatAssociate}
{\em To every disk $\mcal{D} \in \mcal{F} \cap \mcal{T}^{(3)}$, properly embedded in some tetrahedron $\tau \in \mcal{T}$, a} $\mcal{T}$-flat associate {\em disk, denoted by $\widehat{\mcal{D}}$, is assigned according to the following:
\begin{itemize}
\item If $\mcal{D}$ is an elementary disk then $\partial \mcal{D}$ is a collection of either $3$ or $4$ normal arcs $\gamma_i$ each is properly embedded in some $2$-face $\mathfrak{f}_i$. Replace each of the arcs $\gamma_i$ by the geodesic segment in $\mathfrak{f}_i$ that connects the end points of $\gamma_i$, and take $\widehat{\mcal{D}}$ to be the least area disk spanned by these geodesics inside the tetrahedron $\tau$. Note that if $\mcal{D}$ is a quad then $\widehat{\mcal{D}}$ may not be flat even when $\mcal{M}$ is a Euclidean $3$-manifold. This is because the four vertices of $\mcal{D}$ need not be coplanar in $\mcal{M}$. However, in this case $\widehat{\mcal{D}}$ can be conned to the center of math of the four vertices in a standard way.
\item Let $\mcal{D}$ be a non normal disk of in $\mcal{F} \cap \mcal{T}^{(3)}$. According to Theorem \ref{thm:NonNormalAreTame} $\mcal{D}$ can be parameterized by a single face of $\mcal{T}$. Take the parameter domain as the $\mcal{T}$-flat associate of $\mcal{D}$. In case $\mcal{D}$ has corners then the puncture that is formed at each corner is capped off by plugging a triangle $\Delta$ as in Figure \ref{fg:FlatAssociateCorner} below.
\begin{fg}[H]
\[ 
	\setlength{\unitlength}{0.5\standardunitlength}
	\begin{array}{c}  \hspace{-1.7mm}
		\raisebox{-8pt}{\begin{picture}(0,0)%
\includegraphics{FlatAssociateCorner.pstex}%
\end{picture}%
\setlength{\unitlength}{1184sp}%
\begingroup\makeatletter\ifx\SetFigFont\undefined%
\gdef\SetFigFont#1#2#3#4#5{%
  \reset@font\fontsize{#1}{#2pt}%
  \fontfamily{#3}\fontseries{#4}\fontshape{#5}%
  \selectfont}%
\fi\endgroup%
\begin{picture}(5203,5199)(2164,-6823)
\put(4501,-2986){\makebox(0,0)[lb]{\smash{{\SetFigFont{5}{6.0}{\rmdefault}{\mddefault}{\updefault}{\color[rgb]{0,0,0}$\Delta$}%
}}}}
\end{picture}%
 }
		\hspace{-1.9mm}
	\end{array}
 \]
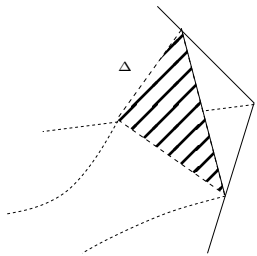
\caption{flat associate near a corner}\label{fg:FlatAssociateCorner}
\end{fg}
\end{itemize}
The {\em $\mcal{T}$-flat associate} of the entire surface $\mcal{F}$ is the union of all $\mcal{T}$-flat associate disks as above. It is denoted by $\widehat{\mcal{F}}$.
}
\end{df}
{\noindent Figure} \ref{fg:FlatAssociates} illustrates the possible flat associated disks according to the definition above. In each of the parts $(A)$ to $(D)$ of the figure, $(a)$ represents a disk $\mcal{D} \in \mcal{F} \cap \mcal{T}^{(3)}$, and $(b)$ represents its $\mcal{T}$-flat associated disk $\widehat{\mcal{D}}$. In parts $(A)$ and $(B)$ flat associated disks that correspond to elementary $3$-gon and $4$-gon are shown. Part $(C)$ shows the flat associated disk of a non normal disk without corners, and in $(D)$ the flat associated disk of a non normal disk with corner is illustrated.
\begin{fg}[H]
\[ 
	\setlength{\unitlength}{0.5\standardunitlength}
	\begin{array}{c}  \hspace{-1.7mm}
		\raisebox{-8pt}{\begin{picture}(0,0)%
\includegraphics{FlatAssoc3gon.pstex}%
\end{picture}%
\setlength{\unitlength}{1973sp}%
\begingroup\makeatletter\ifx\SetFigFont\undefined%
\gdef\SetFigFont#1#2#3#4#5{%
  \reset@font\fontsize{#1}{#2pt}%
  \fontfamily{#3}\fontseries{#4}\fontshape{#5}%
  \selectfont}%
\fi\endgroup%
\begin{picture}(12090,4983)(2614,-5557)
\put(5176,-5461){\makebox(0,0)[lb]{\smash{{\SetFigFont{9}{10.8}{\rmdefault}{\mddefault}{\updefault}{\color[rgb]{0,0,0}$(A)$  flat associate of an elemetray $3$-gon}%
}}}}
\end{picture}%
 }
		\hspace{-1.9mm}
	\end{array}
 \]
\[ 
	\setlength{\unitlength}{0.5\standardunitlength}
	\begin{array}{c}  \hspace{-1.7mm}
		\raisebox{-8pt}{\begin{picture}(0,0)%
\includegraphics{FlatAssoc4gon.pstex}%
\end{picture}%
\setlength{\unitlength}{1973sp}%
\begingroup\makeatletter\ifx\SetFigFont\undefined%
\gdef\SetFigFont#1#2#3#4#5{%
  \reset@font\fontsize{#1}{#2pt}%
  \fontfamily{#3}\fontseries{#4}\fontshape{#5}%
  \selectfont}%
\fi\endgroup%
\begin{picture}(12090,5058)(2614,-5632)
\put(5776,-5536){\makebox(0,0)[lb]{\smash{{\SetFigFont{9}{10.8}{\rmdefault}{\mddefault}{\updefault}{\color[rgb]{0,0,0}$(B)$  flat associate of a $4$-gon}%
}}}}
\end{picture}%
 }
		\hspace{-1.9mm}
	\end{array}
 \]
\[ 
	\setlength{\unitlength}{0.5\standardunitlength}
	\begin{array}{c}  \hspace{-1.7mm}
		\raisebox{-8pt}{\begin{picture}(0,0)%
\includegraphics{FlatAssocNoCorner.pstex}%
\end{picture}%
\setlength{\unitlength}{1973sp}%
\begingroup\makeatletter\ifx\SetFigFont\undefined%
\gdef\SetFigFont#1#2#3#4#5{%
  \reset@font\fontsize{#1}{#2pt}%
  \fontfamily{#3}\fontseries{#4}\fontshape{#5}%
  \selectfont}%
\fi\endgroup%
\begin{picture}(11574,4983)(2614,-5557)
\put(5026,-5461){\makebox(0,0)[lb]{\smash{{\SetFigFont{9}{10.8}{\rmdefault}{\mddefault}{\updefault}{\color[rgb]{0,0,0}$(C)$  flat associate of a non normal disk without corners}%
}}}}
\end{picture}%
 }
		\hspace{-1.9mm}
	\end{array}
 \]
\[ 
	\setlength{\unitlength}{0.5\standardunitlength}
	\begin{array}{c}  \hspace{-1.7mm}
		\raisebox{-8pt}{\begin{picture}(0,0)%
\includegraphics{FlatAssocWithCorner.pstex}%
\end{picture}%
\setlength{\unitlength}{1973sp}%
\begingroup\makeatletter\ifx\SetFigFont\undefined%
\gdef\SetFigFont#1#2#3#4#5{%
  \reset@font\fontsize{#1}{#2pt}%
  \fontfamily{#3}\fontseries{#4}\fontshape{#5}%
  \selectfont}%
\fi\endgroup%
\begin{picture}(11574,4833)(2614,-5407)
\put(4951,-5311){\makebox(0,0)[lb]{\smash{{\SetFigFont{9}{10.8}{\rmdefault}{\mddefault}{\updefault}{\color[rgb]{0,0,0}$(D)$  flat associate of a non normal disk with corner}%
}}}}
\end{picture}%
 }
		\hspace{-1.9mm}
	\end{array}
 \]
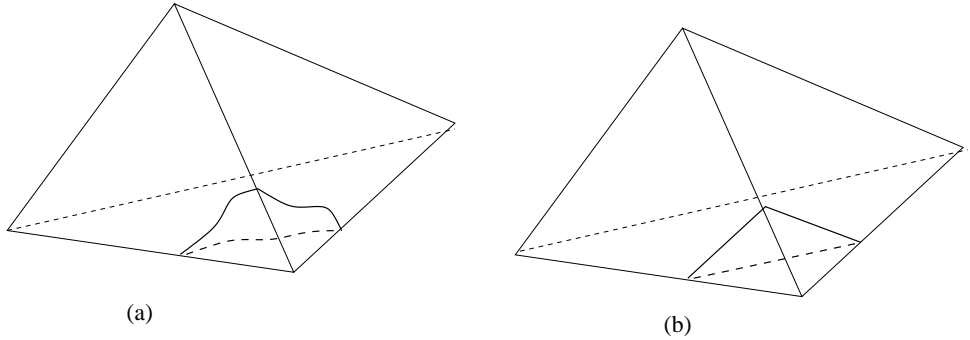
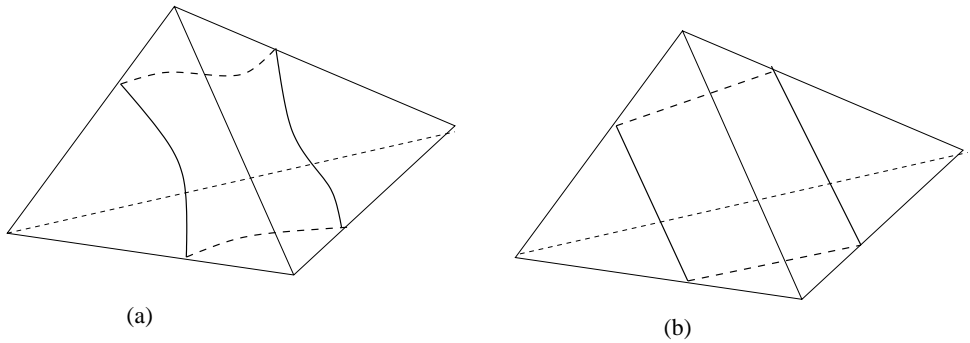
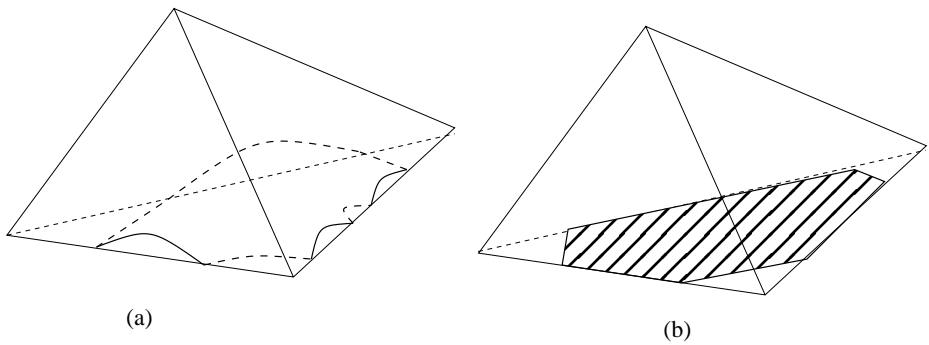
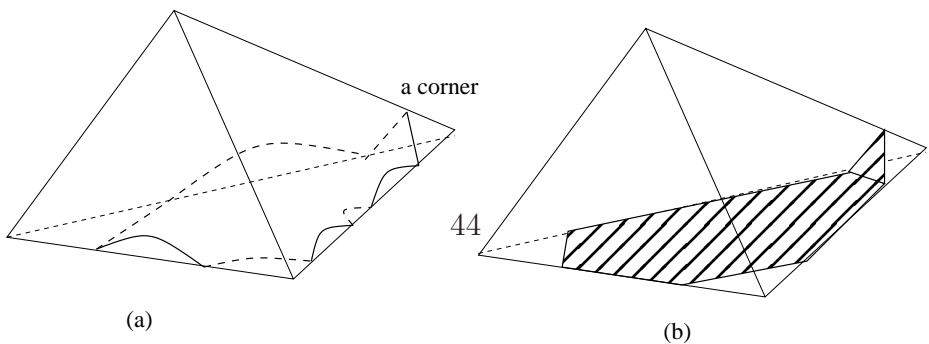
\caption{flat associated disks}\label{fg:FlatAssociates}
\end{fg}
{\noindent Flat} associated disks in adjacent tetrahedra may be glued to each other either along a geodesic segment that correspond to a normal arc that is contained inside a common $2$-face of the two tetrahedra or along a piece of an edge that belongs to these tetrahedra, where this piece corresponds to a bent curve on the common $2$-face of these two tetrahedra. An illustration of such gluing is illustrated in the following Figure \ref{fg:FlatAssociateGluing}.
\begin{fg}[H]
\[ 
	\setlength{\unitlength}{0.5\standardunitlength}
	\begin{array}{c}  \hspace{-1.7mm}
		\raisebox{-8pt}{\begin{picture}(0,0)%
\includegraphics{FlatAssocGluing.pstex}%
\end{picture}%
\setlength{\unitlength}{1973sp}%
\begingroup\makeatletter\ifx\SetFigFont\undefined%
\gdef\SetFigFont#1#2#3#4#5{%
  \reset@font\fontsize{#1}{#2pt}%
  \fontfamily{#3}\fontseries{#4}\fontshape{#5}%
  \selectfont}%
\fi\endgroup%
\begin{picture}(6992,5499)(1564,-6298)
\end{picture}%
 }
		\hspace{-1.9mm}
	\end{array}
 \]
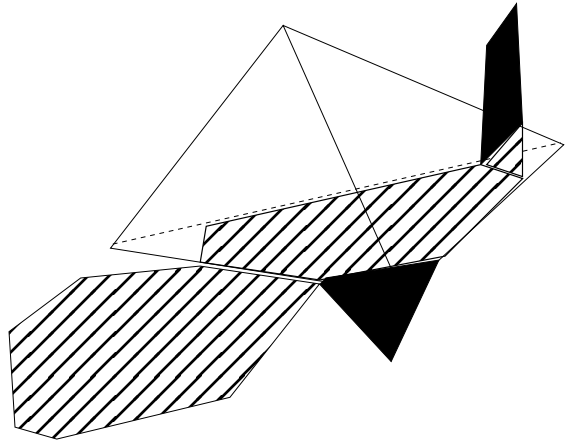
\caption{gluing of flat associated disks. Bold pieces correspond to elementary disks and dashed pieces correspond to non normal disks intersecting along bent curve}\label{fg:FlatAssociateGluing}
\end{fg}
\begin{rem}
{\em At this point the notion {\em flat associate} is somewhat abuse of notation, since $\widehat{\mcal{F}}$ is not really flat or even piecewise flat. In a consecutive section a $p\ell$ version of $\widehat{\mcal{F}}$ will be introduced.
}
\end{rem}

{\noindent Theorem} \ref{thm:AppxThm} is stated below.
\renewcommand{\thethm}{\arabic{thm}}
\setcounter{thm}{4}
\begin{thm}\label{thm:AppxThm}
Let $\mcal{M}$ be a closed, orientable, irreducible, Euclidean or hyperbolic $3$-manifold. Let $\mcal{F}$ be a closed orientable incompressible least area surface embedded in $\mcal{M}$. Then there exists a sequence $\large{\{}\mcal{T}_n\large{\}}$ of $\varphi$-fat triangulations for some fixed $\varphi$, so that $\lambda_n = \lambda(\mcal{T}_n)\searrow 0$ as $n$ goes to $\infty$, and there exists a sequence $\large{\{}\mcal{F}_n\large{\}}$ of surfaces, so that for each $j$, $\mcal{F}_j$ is a $\mcal{T}_j$-flat associate of $\mcal{F}$, and such that the sequence $\large{\{}\mcal{F}_n\large{\}}$ converges to $\mcal{F}$.
\end{thm}
\renewcommand{\thethm}{\thesection.\arabic{thm}}
\setcounter{thm}{1}
\subsection{Preliminary Lemmas}
{\noindent Theorem} \ref{thm:AppxThm} will be deduced from Lemmas \ref{lem:AnglesGoZero}, \ref{lem:AlmostParallelNormals} and \ref{lem:NoCylinders} below.

{\noindent Let} $\mcal{D}$ be some non-normal disk contained in some tetrahedron $\tau \in \mcal{T}$. Let $\mathfrak{f}$ be the $2$-face of $\tau$ that $\mcal{D}$ is parameterized by, and let $\gamma$ denote a bent curve of $\mcal{D}$ not contained in $\mathfrak{f}$, so that the points $p$ and $q$ are the end points of $\gamma$ inside an edge $e$ of $\mathfrak{f}$. Let $\widehat{\mcal{D}} \subset \mathfrak{f}$ be the $\mcal{T}$ flat associate of $\mcal{D}$, and denote the parametrization of $\mcal{D}$ by $\mathfrak{p}(x, y)$, where $(x, y) \in \widehat{\mcal{D}}$ are the parametrization coordinates. Without loss of generality we can assume that the origin of these coordinates is placed at $p$. Suppose $\mcal{M}$ is given some orientation. This induces an orientation on each connected component of $\mcal{F} \cap \mcal{T}^{(2)}$, which in turn, induces an order on the set of intersection points of $\mcal{F} \cap \mcal{T}^{(1)}$ along each such connected component. In particular, assume an order along the component $\gamma$, so that $p$ is the point preceding $q$ in this ordering. Denote by $\bar{t}_p\gamma, \bar{t}_pe$ the tangent vectors at $p$ in the directions of $\gamma$ and $e$ respectively, and assume $\bar{t}_p\gamma$ and $\bar{t}_pe$ meet at angle $\theta_p$. Similar notations are used at the point $q$. In the following Lemma we show that as the mesh of $\mcal{T}$ gets very small $\theta_p$ also gets small

\begin{lem}\label{lem:AnglesGoZero}
For every $\epsilon > 0$ there exists $\delta >0 $ that depends only on $K_{\mcal{M}}$ and $\epsilon$, such that if $\lambda(\mcal{T}) < \delta$ then $\theta_p < \epsilon$.
\end{lem}

\begin{proof}
{\noindent Since} $\theta_p$ and $\theta_q$ have opposite signs the turning of angles between $p$ and $q$, $\triangle \theta  = |\theta_q - \theta_p|$ is at least $\theta_p$. Since $\gamma$ is a bent arc it is completely contained in a single $2$-face of $\tau$, and from Theorem \ref{thm:NonNormalAreTame} it can be parameterized by the geodesic segment of $e$ which goes from $p$ to $q$, denoted by $e_{pq}$. Assume this parametrization is given by arc length with parameter $\xi$. Then we have that:
\begin{eq} \label{eq:AngleDifference}
\theta_p \leq \triangle \theta \leq |\int_{e_{pq}} II d\xi| \leq \int_{e_{pq}} |II| d\xi \leq C \len(e_{pq}) \leq C \lambda(\mcal{T}) \;,
\end{eq}
{\noindent where} $II$ is the second fundamental form of $\mcal{F}$ , and $C$ is the curvature bound from \cite{sch}. This completes then proof of Lemma \ref{lem:AnglesGoZero}.
\end{proof}
{\noindent Let} $z \in \mcal{D}$ be a point on the non-normal disk $\mcal{D}$, and suppose $z$ is the image of the point $\widehat{z} \in \widehat{\mcal{D}}$ under the parametrization $\mathfrak{p}$ of $\mcal{D}$. Let $N_z\mcal{D}, N_{\widehat{z}}\widehat{\mcal{D}}$ denote the unit normal vectors of $\mcal{D}$ and $\widehat{\mcal{D}}$ at $z$ and $\widehat{z}$ respectively. Let $\measuredangle(\vec{v}, \vec{w})$ denote the angle between the two vectors $\vec{v}$ and $\vec{w}$ in the tangent space of $\mcal{M}$. From Lemma \ref{lem:AnglesGoZero} we deduce the following lemma.

\begin{lem}\label{lem:AlmostParallelNormals}
For every $\epsilon > 0$ there exists $\delta > 0$ that depends only on $K_{\mcal{M}}$ and $\epsilon$, so that if $\lambda< \delta$ then $\measuredangle(N_z\mcal{D}, N_{\widehat{z}}\widehat{\mcal{D}}) < \epsilon$.
\end{lem}
\begin{proof}
From Lemma \ref{lem:AnglesGoZero} we have that there is $\rho > 0$ such that if $\lambda(\mcal{T}) < \rho$ then, $\measuredangle(N_p\mcal{D}, N_{\widehat{p}}\widehat{\mcal{D}}) < \frac{\epsilon}{3}$. In addition, from the Schoen bound on the curvature of $\mcal{D}$ there exists $\rho_1 > 0$ so that for each $z \in \mcal{D}$, within the $\rho_1$-neighborhood of $p$, if $\Gamma$ is a geodesic curve from $p$ to $z$ in $\mcal{D}$ then:
\begin{eq}\label{eq:SmallAnglesBetweenNormals}
|\measuredangle(N_p\mcal{D}, N_z\mcal{D})| = |\int_{\Gamma} II d\xi| \leq \int_{\Gamma}| II d\xi| \leq C\len(\Gamma) \leq C^2d_{\mcal{M}}(p, z) \leq C^2\rho_1 < \frac{\epsilon}{3} \;.
\end{eq}
{\noindent The} same holds for $\measuredangle(N_{\widehat{p}}\widehat{\mcal{D}}, N_{\widehat{z}}\widehat{\mcal{D}})$ as well, since the Schoen bound $C$ depends only in $K_{\mcal{M}}$, and can be used both for $\mcal{F}$ and for $\widehat{\mcal{D}}$.
All in all we get that
\begin{eq} \label{eq:CloseNormalVectors}
\measuredangle(N_z\mcal{D}, N_{\widehat{z}}\widehat{\mcal{D}}) \leq \measuredangle(N_z\mcal{D}, N_p\mcal{D}) + \measuredangle(N_p\mcal{D}, N_{\widehat{p}}\widehat{\mcal{D}}) + \measuredangle(N_{\widehat{p}}\widehat{\mcal{D}}, N_{\widehat{z}}\widehat{\mcal{D}}) < \epsilon
\end{eq}
The proof of Lemma \ref{lem:AlmostParallelNormals} is thus completed.
\end{proof}

{\noindent Lemmas} \ref{lem:AnglesGoZero} and \ref{lem:AlmostParallelNormals} hold also if $\mcal{D}$ is an elementary disk, and $\widehat{\mcal{D}}$ denotes its $\mcal{T}$-flat associate. This is because by Definition \ref{df:FlatAssociate} both disks have the same vertices on $\mcal{T}^{(1)}$, so they intersect at least at three points, so the same analysis as in the proofs above can be done.

\begin{df}\label{df:NonNormalCylinder}
{\em
Let $\tau_1$ and $\tau_2$ be two tetrahedra that share a common $2$-face $\mathfrak{f}$. A} cylinder {\em is a pair of two non normal disks $\mcal{D}_1$ and $\mcal{D}_2$ embedded in $\tau_1$ and $\tau_2$ respectively, so that there exists two bent curves $\gamma_1 \subset \partial \mcal{D}_1$ and $\gamma_2 \subset \partial \mcal{D}_2$, so that $\gamma_1 \cap \gamma_2 = \partial \gamma_1 = \partial \gamma_2$, and $\mcal{D}_1$ and $\mcal{D}_2$ are parameterized by the same $2$-face. Figure \ref{fg:Cylinder} illustrates the local picture of a cylinder around the union of $\gamma_1$ and $\gamma_2$.
}
\end{df}
\begin{fg}[H]
\[ 
	\setlength{\unitlength}{0.5\standardunitlength}
	\begin{array}{c}  \hspace{-1.7mm}
		\raisebox{-8pt}{\input NonNormalCylinder.tex }
		\hspace{-1.9mm}
	\end{array}
 \]
\caption{a cylinder}\label{fg:Cylinder}
\end{fg}

\begin{lem}\label{lem:NoCylinders}
Let $\mcal{F}$ be a closed, incompressible least area surface embedded in the closed irreducible triangulated $3$-manifold $(\mcal{M}, \mcal{T})$. There exists a constant $\rho$ that depends only on $K_{\mcal{M}}$ and $inj_{\mcal{M}}$ so that if $\lambda(\mcal{T}) < \rho$ then $\mcal{F}$ is quasi normal with respect to $\mcal{T}$ and there are no cylinders in $\mcal{F} \cap \mcal{T}^{(3)}$.
\end{lem}
\begin{proof}
The quasi normal part was already proved in Theorem \ref{thm:LeastAreaIsNormal}. Suppose there exists a cylinder in $\mcal{F} \cap \mcal{T}^{(3)}$, and let $\gamma_1, \gamma_2$ be the two bent curves as in Definition \ref{df:NonNormalCylinder}. If $\lambda(\mcal{T}) < \frac{1}{2\sqrt{C}}$, where $C$ is the Schoen curvature bound then $\gamma \cup \gamma_2$ is a simple closed curve that is contained in a ball of radius $\lambda(\mcal{T}) < \frac{1}{2\sqrt{C}}$. As a result, there exists at least one point in $\gamma_1 \cup \gamma_2$ at which its curvature is greater than $C$, and this yields a contradiction. This completes the proof.
\end{proof}
\subsubsection{Proof of Theorem \ref{thm:AppxThm}}

\begin{proof}[Proof of Theorem \ref{thm:AppxThm}]
Let $\mcal{T}$ be some given initial triangulation of $\mcal{M}$ of some fatness $\varphi$. Define the sequence  $\large{\{}\mcal{T}_n\large{\}}$ by setting:
\begin{eq}\label{eq:TriangulationTn}
\mcal{T}_j = Med_{(j)}(\mcal{T})
\end{eq}
{\noindent Let} the surface $\widehat{\mcal{F}}_j$ be the $\mcal{T}_j$-flat associate of $\mcal{F}$.\\

{\noindent Since} the median subdivision of $\mcal{T}$ is of the same fatness as $\mcal{T}$, and has mesh of order half of the mesh of $\mcal{T}$ there exists $j_0 > 0 $ so that the $j_0$ median subdivision of $\mcal{T}$ ,$\mcal{T}_{j_0}$, is also $\varphi$-fat while its mesh is smaller than the constant $\Psi(\varphi)$ obtained in Theorem \ref{thm:LeastAreaIsNormal} so that $\mcal{F}$ is quasi normal with respect to $\mcal{T}_{j}$, for all $j > j_0$. Hence, without loss of generality we may assume that $\mcal{F}$ is quasi normal with respect the initial triangulation $\mcal{T}$.
\begin{clm}\label{clm:MetricConvergence}
The Hausdorff distance between $\mcal{F}$ and $\widehat{\mcal{F}}_j$ converges to zero as $j \rightarrow \infty$.
\end{clm}
\begin{proof}
{\noindent Let} $\mcal{D}$ be some disk in $\mcal{F} \cap \mcal{T}^{(3)}$, either elementary or non normal, properly embedded in some tetrahedron $\tau_j \in \mcal{T}_j$, and let $\widehat{\mcal{D}}_j$ be its $\mcal{T}_j$-flat associate disk. Let $p, \widehat{p}$ be arbitrary points in $\mcal{D}, \widehat{\mcal{D}}$ respectively. Since $\mcal{D}$ intersects $\widehat{\mcal{D}}$ at least at two points $v, w$ along the $1$-skeleton of $\mcal{T}_j$ ($v, w$ are either the end points of a bent curve or the vertices of elementary disks), the following holds:
\begin{eq}
d(p,\widehat{p}) \leq d(p,v) + d(v,\widehat{p}) \leq diam(\tau_j) + diam(\tau_j) \;.
\end{eq}
Hence when passing to the maximum over all disks $\mcal{D}$ one obtains:
\begin{eq}
d_H(\mcal{D}, \widehat{\mcal{D}}) \leq 2 \lambda(\mcal{T}_j) = O(\frac{1}{2^{j-1}})\lambda(\mcal{T})
\end{eq}
This completes the proof of Claim \ref{clm:MetricConvergence}.
\end{proof}
{\noindent Claim} \ref{clm:MetricConvergence} proofs that the sequence of surfaces $\large{\{}\widehat{\mcal{F}}_j\large{\}}$ converges to $\mcal{F}$ in the metric sense, and from Lemmas \ref{lem:AnglesGoZero} and \ref{lem:AlmostParallelNormals} it follows that if $\{x_i \in \widehat{\mcal{F}}_i\}$ is a sequence of points so that $x_i \rightarrow x$ for some point $x \in \mcal{F}$, then $N_{x_i}\widehat{\mcal{F}}_i \rightarrow N_x\mcal{F}$ and Theorem \ref{thm:AppxThm} is deduced.
\end{proof}

\begin{cor}\label{cor:AppxAreaDisks}
Under the conditions of Theorem \ref{thm:AppxThm} we have that if $\mcal{D}$ is a disk in $\mcal{F} \cap \mcal{T}_i$, and $\widehat{\mcal{D}}_i$ is its $\mcal{T}_i$-flat associate then
\[ \are(\widehat{\mcal{D}}_i) \rightarrow \are(\mcal{D}) \;. \]
\end{cor}
\begin{proof}
Following Thm \ref{thm:AppxThm}, Lemma \ref{lem:MorvanAreaAppx} can be applied, according to which for every $\epsilon > 0$ there exists $j_0$ that depends on $\epsilon$, so that for each $j > j_0$ and for every disk $\mcal{D} \in \mcal{F} \cap \mcal{T}_j^{(3)}$ the following holds:
\begin{eq}
|\are(\mcal{D}) - \are(\widehat{\mcal{D}}_j)| < \frac{\epsilon}{2^{j_0}} \;.
\end{eq}
This completes the proof.
\end{proof}
{\noindent In} order to deduce that $\are(\mcal{F}) = \lim_{j \rightarrow \infty} \are(\widehat{\mcal{F}_j})$ it is necessary to know that every flat associate disk $\widehat{\mcal{D}}_j$ parameterizes a single non normal disk in $\mcal{F} \cap \mcal{T}^{(3)}$. Since this is not guaranteed, the area of every flat associate disk $\widehat{\mcal{D}}_j$ will be accounted for with its multiplicity. That is, according to the number of non normal disks in $\mcal{F} \cap \mcal{T}^{(3)}$ that are projected into it. From compactness there are always finitely many such non normal disks. The sum of areas of all $\mcal{T}_i$-flat associated disks, along with their multiplicities, will be called the {\em multiple area} of the flat associated surface $\widehat{\mcal{F}}_i$ and will be denoted by $\widetilde{\are(\widehat{\mcal{F}_i})}$. Then the following holds:
\begin{cor}\label{cor:AppxArea}
Under the conditions of Theorem \ref{thm:AppxThm}
\[ \are(\mcal{F}) = \lim_{j \rightarrow \infty} \widetilde{\are(\widehat{\mcal{F}_i})} \;. \]
\end{cor}
\begin{proof}
From Lemma \ref{cor:AppxAreaDisks} it follows that taking the sum of areas over all disks $\widehat{\mcal{D}}$ with their multiplicities gives that
\begin{eq}
|\are(\mcal{F}) - \widetilde{\are(\widehat{\mcal{F}_i})}| < \sum \frac{\epsilon}{2^{j_0}} < \epsilon\;.
\end{eq}
This completes the proof of the corollary.
\end{proof}

\subsection{Piecewise Linear Approximation in $\mathbb{R}^n$}
Recalling that the starting point for this work is the Jaco-Rubinstein $p\ell$ version of least area and minimal surfaces this subsection is devoted to prove a $p\ell$ version of Theorem \ref{thm:AppxThm}. This will be achieved by considering an isometric embedding of the manifold $\mcal{M}$ in some $\mathbb{R}^n$ for some large enough $n$, as guaranteed by the Nash embedding theorem, see \cite{nas}. The $p\ell$ approximation that will be obtained is therefore an approximation in $\mathbb{R}^n$ rather than inside the manifold $\mcal{M}$, as done in the previous subsection. Such a $p\ell$ version of Theorem \ref{thm:AppxThm} is further desirable since it is the common viewpoint in the variety of potential applications for ``real world'' problems in areas such as image processing and computer graphics as mentioned in the introduction.\\
{\noindent The} following $p\ell$ version of Theorem \ref{thm:AppxThm} will be proved in this section.

\renewcommand{\thethm}{\arabic{thm}}
\setcounter{thm}{5}
\begin{thm}\label{thm:AppxThmLinear}
Let $\mcal{M}$ be a closed, orientable, irreducible, Euclidean or hyperbolic $3$-manifold isometrically embedded in $\mathbb{R}^n$. Let $\mcal{F}$ be a closed orientable incompressible least area surface embedded in $\mcal{M}$. There exists a sequence $\large{\{}\widehat{\mcal{M}}_n\large{\}}$ of piecewise linear $3$-manifolds embedded in $\mathbb{R}^n$, each homeomorphic to $\mcal{M}$, and there exists a sequence $\large{\{}\widehat{\mcal{F}}_n\large{\}}$ of piecewise linear surfaces, so that each $\widehat{\mcal{F}}_j$ is embedded in $\widehat{\mcal{M}}_j$, and such that $\widehat{\mcal{F}}_j \rightarrow \mcal{F}$ as $j \rightarrow \infty$, with respect to the Hausdorff metric on closed surfaces embedded in $\mathbb{R}^n$.\\
\end{thm}
\renewcommand{\thethm}{\thesection.\arabic{thm}}
\setcounter{thm}{1}
\subsubsection{Preliminary Lemmas Again}
{\noindent Let} $\mcal{T}$ be some $\varphi$-fat triangulation of the Euclidean or hyperbolic manifold $\mcal{M}$, and let $\widehat{\mcal{T}}$ be its linearization imposed by the secant map induced by $\mcal{T}$, as in Definition \ref{df:SecantMapII}. That is, the vertices of $\widehat{\mcal{T}}$ are the same as the vertices of $\mcal{T}$ and in all other dimensions it is given as the linear span over the set of vertices.
\begin{lem}\label{lem:LinearAngleAppxTwoFaces}
Let $\mathfrak{f}$ be a $2$-face of $\mcal{T}$ and let $\widehat{\mathfrak{f}}$ be its linearization in $\widehat{\mcal{T}}$. Let $v$ and $w$ be two common vertices of both $\mathfrak{f}$ and $\widehat{\mathfrak{f}}$ and let $e$ and $\widehat{e}$ denote the edges of $\mathfrak{f}$ and $\widehat{\mathfrak{f}}$ respectively, between $v$ and $w$. Let $\theta_v$ be the angle between $e$ and $\widehat{e}$ at $v$. Then for every $\epsilon > 0$ there exists $\delta > 0$ that depends only on $K_{\mcal{M}}$ and $\epsilon$, so that if $\lambda(\mcal{T}) < \delta$ then  $|\theta_v| < \epsilon$.
\end{lem}
\begin{proof}
Recall that $\mathfrak{f}$ is a least area disk in $\mcal{M}$, as $\mcal{T}$ is a least area triangulation of $\mcal{M}$. Let $C$ denote the Schoen bound with respect to the metric of $\mcal{M}$. In particular, this bound holds for the $2$-face $\mathfrak{f}$. In order to prove the lemma we would like have such a bound on the second fundamental form of $\mathfrak{f}$ also with respect to the Euclidean metric of $\mathbb{R}^n$. If $\mcal{M}$ is Euclidean then $C$ is the desired bound. If $\mcal{M}$ is hyperbolic manifold then in Lemma $2.6$ of \cite{clr} it is proved that the second fundamental form of $\mathfrak{f}$ is bounded also with respect to the Euclidean metric of $\mathbb{R}^n$, by $8C$. Using this for example in the hyperbolic case, the following holds:
\begin{eq} \label{eq:SmallAngleBetweenEdges}
|\theta_v| \leq |\theta_w - \theta_v| \leq |\int_{e_{vw}} II_{Euc} d\xi| \leq \int_{e_{vw}} |II_{Euc}| d\xi \leq 8C \len_{Euc}(e_{vw}) \;,
\end{eq}
where $II_{Euc}, \len_{Euc}$ denote, respectively, the second fundamental form and the length function with respect to the Euclidean metric. In addition, if $\len_{hyp}$ denotes the length with respect to the hyperbolic metric of $\mcal{M}$ then:
\begin{eq}\label{eq:FromEucLengthToHyp}
\len_{Euc}(e_{vw}) < K_{\mcal{M}}\len_{hyp}(e_{vw}) \leq K_{\mcal{M}}\lambda(\mcal{T}) < C\lambda(\mcal{T})\; .
\end{eq}
{\noindent In} the Euclidean case similar inequalities will hold. From Equations \ref{eq:SmallAngleBetweenEdges} and \ref{eq:FromEucLengthToHyp} the lemma is deduced.
\end{proof}

{\noindent Let} $\mathfrak{f}$ and $\widehat{\mathfrak{f}}$ be as in Lemma \ref{lem:LinearAngleAppxTwoFaces}. For every point $p \in \mathfrak{f}$ and every point $\widehat{p}$ in $\widehat{\mathfrak{f}}$ let $N_p\mathfrak{f}$ and $N_{\widehat{p}}\widehat{\mathfrak{f}}$ be the unit normal vectors to $\mathfrak{f}$ and $\widehat{\mathfrak{f}}$ respectively. Similarly to Lemma \ref{lem:AlmostParallelNormals} we have the following lemma:

\begin{lem}\label{lem:LinearAppxOfNormalsTwoFaces}
For every $\epsilon > 0$ there exists $\delta > 0$ that depends only on $\epsilon$ and $K_{\mcal{M}}$, such that if $\lambda(\mcal{T}) < \delta$ then $\measuredangle(N_p\mathfrak{f}, N_{\widehat{p}}\widehat{\mathfrak{f}}) < \epsilon$.
\end{lem}
\begin{proof}
According to Theorem \ref{thm:Munkres1} there exists $\delta > 0$ that depends only on $\varphi$ such that if $\lambda(\mcal{T}) < \delta$ then $\widehat{\mathfrak{f}}$ is contained in the $\epsilon$ tubular neighborhood of $\mathfrak{f}$. Let $\Gamma$ be a geodesic curve in $\mathfrak{f}$ connecting the vertex $v$ of $\mathfrak{f}$ to $p$. Then there exists $\delta_1 >0$ that depends only on $\epsilon$ and $K_{\mcal{M}}$, such that:
\begin{eq}\label{eq:SmallAnglesBetween2Faces}
|\measuredangle(N_p\mathfrak{f}, N_v\mathfrak{f})| \leq \int_{\Gamma}| II_{Euc} d\xi| \leq 8C\len(\Gamma) \leq 8C^2d_{\mcal{M}}(p, v) \leq 8C^2\lambda(\mcal{T}) < \frac{\epsilon}{2} \;.
\end{eq}

{\noindent In} addition, since $v$ is also a vertex of $\widehat{\mathfrak{f}}$, and since $\widehat{\mathfrak{f}}$ is a planar $2$-face in $\mathbb{R}^n$, the normal vectors to $\widehat{\mathfrak{f}}$ at $\widehat{p}$ and at $v$ are the same, hence from Lemma \ref{lem:LinearAngleAppxTwoFaces} there exists $\delta_2 > 0$ that depends only on $\epsilon$ and $K_{\mcal{M}}$ so that if $\lambda(\mcal{T}) < \delta_2$ then $|\measuredangle(N_p\mathfrak{f}, N_{\widehat{p}}\widehat{\mathfrak{f}})| \leq \frac{\epsilon}{2}$. Combining this with Equation \ref{eq:SmallAnglesBetween2Faces} proves the lemma.
\end{proof}
{\noindent Lemmas} \ref{lem:LinearAngleAppxTwoFaces} and \ref{lem:LinearAppxOfNormalsTwoFaces} concern the approximation of $2$-faces of $\mcal{T}$ by $2$-faces of $\widehat{\mcal{T}}$, and will later be used in the proof of Theorem \ref{thm:AppxThmLinear}. The following construction and lemma are also needed.
\subsubsection{Building the Approximating Surfaces}
\begin{df}\label{df:ElementaryPlFlatAssociate}
{\em
Let $\tau$ be a tetrahedron of $\mcal{T}$ and let $\widehat{\tau}$ be the corresponding $p\ell$-tetrahedron of $\widehat{\mcal{T}}$. Let $\mcal{D}$ be some elementary disk in $\mcal{F} \cap \tau$, and let $\widehat{\mcal{D}}$ denote the corresponding $p\ell$-disk in $\widehat{\mcal{S}} \cap \widehat{\tau}$, built as follows:
\begin{enum}[(a)]
\item \label{IntersectionPoints} First the intersection points of $\widehat{\mcal{D}}$ with the edges of $\widehat{\tau}$ will be explicitly described. Let $\mcal{E} \in \mcal{T}^{(1)}$ be an edge, and let $x \in \mcal{E} \cap \mcal{F}$ be an internal point of $\mcal{E}$ as shown in Figure \ref{fig:SecantApproxTetra} below.
\begin{fg}[H]
\[ 
	\setlength{\unitlength}{0.5\standardunitlength}
	\begin{array}{c}  \hspace{-1.7mm}
		\raisebox{-8pt}{\input PlApproxTetra.tex }
		\hspace{-1.9mm}
	\end{array}
 \]
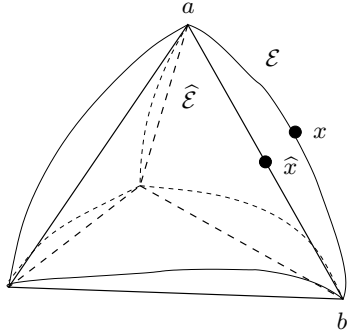
\caption{\small a tetrahedron of $\mcal{T}$ ({\em thin curved lines}) and its secant $p\ell$-approximation ({\em bold lines}).}
\label{fig:SecantApproxTetra}
\end{fg}
{\noindent Define} the corresponding point $\widehat{x}$ to be the unique point that has the same relative distance from the vertex $a$ with respect to the length of $\widehat{\mcal{E}}$ as the relative distance of $x$ from $a$ along $\widehat{\mcal{E}}$ with respect to the length of $\mcal{E}$. Analytically this means that $\widehat{x}$ is the point on the {\em straight} edge $\widehat{\mcal{E}} \in \widehat{\mcal{T}}^{(1)}$, that satisfies the equality,
\begin{eq}\label{eq:ProjectionEdje}
\frac{d(a, \widehat{x})}{d(a, b)} = \frac{d_{\mcal{E}}(a, x)}{d_{\mcal{E}}(a, b)} \;\; ,
\end{eq}
\item \label{PlanarNormalArcs}The rest of the disk $\widehat{\mcal{D}}$ is taken as the linear span over the points defined above.
\item If $\mcal{D}$ is a non-normal disk with respect to $\mcal{T}$ then recall its corresponding $\mcal{T}$-flat associate that was built in the previous subsection. The corresponding disk $\widehat{\mcal{D}}$ is built similarly to the case of elementary disks above with the exception that $\mcal{D}$ is replaced by its $\mcal{T}$-flat associate.
\end{enum}
}
\end{df}
\begin{rem}
{\em
{\noindent Note} that this construction is well defined with respect to tetrahedra in the sense that, for each point of intersection $x \in \mcal{S} \cap \mcal{T}^{(1)}$, the same point $\widehat{x}$ is obtained, no matter which of the tetrahedra adjacent to $\mcal{E} \in \mcal{T}^{(1)}$ is considered. However, when $\mcal{D}$ is a quadrilateral, one must chose between two options. This is due to the fact that the four intersection points that serve as the vertices of  $\widehat{\mcal{D}}$, may not be coplanar. If this is indeed the case, then two quadrilaterals can be built as illustrated in the following Figure \ref{fg:NonUniqueQuad}.
\begin{fg}[H]
\[ 
	\setlength{\unitlength}{0.5\standardunitlength}
	\begin{array}{c}  \hspace{-1.7mm}
		\raisebox{-8pt}{\input NonUniquePlQuad.tex }
		\hspace{-1.9mm}
	\end{array}
 \]
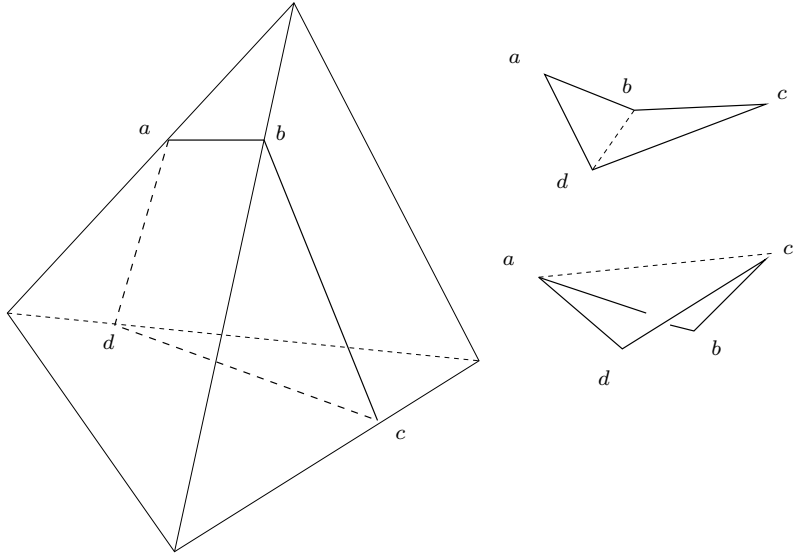
\caption{\small vertices $a, b, c, d$ may not be coplanar thus, these four points can span two quads.}\label{fg:NonUniqueQuad}
\end{fg}
{\noindent Since} these two quadrilaterals have common edges, and in addition their interiors are distinct $p\ell$ disks, then they are isotopic to each other inside the tetrahedron $\widehat{\tau}$. Therefore, up to isotopy, it is possible to choose arbitrarily between them. An explicit choice can be made by taking the quadrilateral of least area between the two. If the areas of the two are the same, then arbitrarily choose one.
}
\end{rem}
{\noindent As} in the context of Lemma \ref{lem:LinearAppxOfNormalsTwoFaces} let $N_{p}\mcal{D}$ and $N_{\widehat{p}}\widehat{\mcal{D}}$ denote the unit normal vectors to the disks $\mcal{D}$ and $\widehat{\mcal{D}}$ at the points $p \in \mcal{D}$ and $\widehat{p} \in \widehat{\mcal{D}}$ respectively. The following version of Lemma \ref{lem:LinearAppxOfNormalsTwoFaces} holds:

\begin{lem}\label{lem:LinearAppxElementaryDisks}
Let $\epsilon$ be some positive number. There exists $\delta >0$ that depends only on $\epsilon, K_{\mcal{M}}$, and on $\varphi$ so that if $\lambda(\mcal{T}) < \delta$ then $\measuredangle(N_{p}\mcal{D}, N_{\widehat{p}}\widehat{\mcal{D}}) < \epsilon$.
\end{lem}
\begin{proof}
If $\mcal{D}$ is a non normal disk then $\widehat{\mcal{D}}$ is a piece of planar $2$-face of $\widehat{\mcal{T}}$. Let $\overline{\mcal{D}}$ be the $\mcal{T}$-flat associate of $\mcal{D}$, and let $\bar{p}$ be a point in $\overline{\mcal{D}}$. From Lemma \ref{lem:AlmostParallelNormals} there exists $\delta_1 > 0$ which depends only on $\epsilon$ and on $K_{\mcal{M}}$, such that if $\lambda(\mcal{T})$ then $\measuredangle(N_{p}\mcal{D}, N_{\bar{p}}\overline{\mcal{D}}) < \frac{\epsilon}{2}$. From Lemma \ref{lem:LinearAppxOfNormalsTwoFaces} there exists $\delta_2 > 0$ that depends only on $\epsilon$ and on $K_{\mcal{M}}$, such that if $\lambda(\mcal{T}) < \delta_2$ then $\measuredangle(N_{\bar{p}}\overline{\mcal{D}}, N_{\widehat{p}}\widehat{\mcal{D}}) < \frac{\epsilon}{2}$. Hence,
\[ \measuredangle(N_{p}\mcal{D}, N_{\widehat{p}}\widehat{\mcal{D}}) \leq \measuredangle(N_{p}\mcal{D}, N_{\bar{p}}\overline{\mcal{D}}) + \measuredangle(N_{\bar{p}}\overline{\mcal{D}}, N_{\widehat{p}}\widehat{\mcal{D}}) < \epsilon \;, \]
and the lemma is proved.

{\noindent In} case $\mcal{D}$ is an elementary disk then from Lemma \ref{lem:LinearAngleAppxTwoFaces} we have that there exists $\delta_1 > 0$ that depends only on $K_{\mcal{M}}$, $\epsilon$ and such that if $\lambda(\mcal{T}) < \delta_1$ then $|\len(e) - \len(\widehat{e})| < \frac{\epsilon}{3}$ for any edge $e$ of $\mcal{T}$ and its corresponding edge $\widehat{e}$ of $\widehat{\mcal{T}}$. Following this, if $x$ is an intersection point of $\mcal{D}$ with the edge $e$ and $\widehat{x}$ is the corresponding intersection point of $\widehat{\mcal{D}}$ obtained according to the procedure described above, then $\widehat{x}$ is contained in the $\frac{\epsilon}{3}$ neighborhood of $x$ with respect to the Euclidean metric on $\mathbb{R}^n$. Moreover, since $\widehat{\mcal{T}}$ is a $\delta$ approximation of $\mcal{T}$ we have that $\widehat{\mcal{D}}$ is contained in the $\frac{\epsilon}{3}$-neighborhood of $\mcal{D}$ with respect to the Euclidean metric. Since $\mcal{D}$ is a piece of the least area surface $\mcal{F}$ in $\mcal{M}$ its curvature is bounded with respect to either a Euclidean or hyperbolic metric of $\mcal{M}$, and according to Lemma $2.6$ of \cite{clr} it is also bounded with respect to the Euclidean metric of $\mathbb{R}^n$. Combining all the above, and using the similar arguments as in the proof of Lemmas \ref{lem:LinearAngleAppxTwoFaces} and \ref{lem:LinearAppxOfNormalsTwoFaces}, we establish Lemma \ref{lem:LinearAppxElementaryDisks}.
\end{proof}
\subsubsection{Proof of Theorem \ref{thm:AppxThmLinear}}
{\noindent It} is now possible to prove Theorem \ref{thm:AppxThmLinear}.
\begin{proof}[Proof of Theorem \ref{thm:AppxThmLinear}]
Let $\widehat{\mcal{M}}$ be the secant map induced by an isometric embedding of $(\mcal{M}, \mcal{T})$ in $\mathbb{R}^n$. Let $\mcal{T}_j$ be the median subdivision of $\mcal{T}$ and let $\widehat{\mcal{M}}_j$ be the linearization of $(\mcal{M}, \mcal{T}_j)$. Lemmas \ref{lem:LinearAppxOfNormalsTwoFaces} and \ref{lem:LinearAppxElementaryDisks} imply the theorem.
\end{proof}

\section*{Acknowledgement}
Most content of this paper is part of the author's PhD dissertation, under the advisory of Prof. Yoav Moriah, that was submitted to the Technion-IIT graduate school on 2014. The author would like to express his gratitude to Dr. Emil Saucan and to Prof. Joel Hass for their useful comments while working on early versions of this work.

\end{document}